\documentclass[12pt]{amsart}

\usepackage{amsmath}
\usepackage{amssymb}
\usepackage[all]{xy}

\numberwithin{equation}{section}

\newtheorem{thm}[equation]{Theorem}
\newtheorem{lem}[equation]{Lemma}
\newtheorem{prop}[equation]{Proposition}
\newtheorem{cor}[equation]{Corollary}
\newtheorem{rem}[equation]{Remark}         
\newtheorem{exam}[equation]{Example}

\newtheorem{exam-nota}[equation]{Example-Notation}

\newtheorem{dfn}[equation]{Definition}

\newtheorem{dfn-nota}[equation]{Definition-Notation}

\newtheorem{dfn-lem}[equation]{Lemma-Definition}
\newtheorem{dfn-prop}[equation]{Proposition-Definition}

\newcommand{\beqa}{\begin{eqnarray*}}
\newcommand{\eeqa}{\end{eqnarray*}}

\newcommand{\fgtilde}{\widetilde{\mathfrak{g}}}

\newcommand{\wpi}{\widetilde{\pi}}

\newcommand{\fa}{\mbox{${\mathfrak a}$}}

\newcommand{\fg}{\mbox{${\mathfrak g}$}}

\newcommand{\fh}{\mbox{${\mathfrak h}$}}

\newcommand{\fz}{\mbox{${\mathfrak z}$}}
\newcommand{\fm}{\mbox{${\mathfrak m}$}}

\renewcommand{\qed}{\hfill $\Box$}

\renewcommand{\proof}{{\bf Proof.\ \ }}

\newcommand{\C}{\mbox{${\mathbb C}$}}

\newcommand{\Ad}{{\rm Ad}}

\newcommand{\fgl}{\mathfrak{gl}}

\newcommand{\xifij}{\xi_{f_{ij}}}

\newcommand{\dn}{{n\choose 2}}
\newcommand{\dnone}{{n+1\choose 2}}

\newcommand{\ad}{\operatorname{ad}}

\newcommand{\cO}{\mathcal{O}}

\newcommand{\invlim}{\varprojlim}
\newcommand{\dirlim}{\varinjlim}
\newcommand{\OX}{\mathcal{O}_{X}}

\newcommand{\stalk}{\mathcal{O}_{X,x}}

\newcommand{\biveci}{\pi_{\infty}}

\newcommand{\anchori}{\widetilde{\biveci}}

\newcommand{\calL}{\mathcal{L}}
\newcommand{\omegainftylambdatilde}{(\widetilde{\omega_{\infty}})_{\lambda}}

\newcommand{\B}{\mathcal{B}}
\newcommand{\wOX}{\widetilde{\OX}}

\newcommand{\bivecfglambda}{\pi_{\fg^{*},\lambda}}
\newcommand{\bivecfgmu}{\pi_{\fg^{*},\mu}}

\newcommand{\anchorfglambda}{\widetilde{\pi_{\fg^{*}}}_{,\lambda}}
\newcommand{\fnij}{f_{ij}}
\author[M. Colarusso]{Mark Colarusso}
\address{Department of Mathematics, University of Wisconsin--Milwaukee}
\email{colaruss@uwm.edu}
\title{Lie-Poisson Theory for Direct Limit Lie algebras}
\begin{document}
\author[M. Lau]{Michael Lau}
\address{D\'epartement de math\'{e}matiques et de statistique, Universit\'{e} Laval}
\email{Michael.Lau@mat.ulaval.ca}
\subjclass[2010]{14L30, 20G20, 37J35, 53D17, 17B65}
\thanks{Funding from the Natural Sciences and
Engineering Research Council of Canada is gratefully acknowledged.}

\maketitle
\begin{abstract}
In this paper, we develop the fundamentals of Lie-Poisson theory for direct limits $G=\dirlim G_{n}$ of complex algebraic groups $G_{n}$ and their Lie algebras $\fg=\dirlim \fg_{n}$.  We show that $\fg^{*}=\invlim\fg_{n}^{*}$ has the structure of a Poisson provariety and that each coadjoint orbit of $G$ on $\fg^{*}$ has the structure of an ind-variety.  We construct a weak symplectic form on every coadjoint orbit and prove that the coadjoint orbits form a weak symplectic foliation of the Poisson provariety $\fg^{*}$.  
We apply our results to the specific setting of $G=GL(\infty)=\dirlim GL(n,\C)$ and $\fg^{*}= M(\infty)=\invlim \fgl(n,\C)$, the space of infinite complex matrices with
arbitrary entries.  We construct a Gelfand-Zeitlin integrable system on $M(\infty)$, which generalizes the one constructed by 
Kostant and Wallach on $\fgl(n,\C)$.  The system integrates to an action of a direct limit group $A(\infty)$ on $M(\infty)$, whose  
 generic orbits are Lagrangian ind-subvarieties of the corresponding coadjoint orbit of $GL(\infty)$ on $M(\infty)$.
\end{abstract}

\section{Introduction}
The interaction between Lie theory and Poisson geometry plays an important role in much of modern mathematics and mathematical physics; it is of central importance in geometric representation theory, integrable systems, and classical mechanics.  Given a finite-dimensional real or complex Lie group $G$, there is a canonical Lie-Poisson structure on the dual space $\fg^{*}$ of its Lie algebra $\fg$.  The starting point of Lie-Poisson theory is the observation that the symplectic leaves of $\fg^{*}$ are the coadjoint orbits of the identity component $G^{0}$ of $G$ equipped with the Kostant-Kirillov symplectic form.  This symplectic structure is the cornerstone of the orbit method in representation theory and plays an important role in deformation quantization.  

The main goal of this paper is to extend this theory to direct limit Lie algebras.  
Given a direct limit group $G$ with Lie algebra $\fg$, we define an analogous Lie-Poisson structure on the dual space $\fg^{*}$ 
and construct a symplectic foliation using the coadjoint action of $G$ on $\fg^{*}$.  There is an extensive literature concerning direct limit groups and their Lie algebras, root systems, and representations, though there has been little study of Poisson geometry in this context.  See for example, \cite{bahturin-benkart,bahturin-strade,dimitrov-penkov,DP,glockner,loos-neher,neeb,penkov-serganova} and the references therein.  Applications of direct limit groups notably include early work on infinite-dimensional integrable systems, including the KP hierarchy.  See \cite{KR}, for instance.  In particular, when $\fg=\fgl(\infty):=\dirlim \fgl(n,\C)$, we use our new Lie-Poisson structure to define an infinite-dimensional analogue of the Gelfand-Zeitlin integrable system on $\fgl(\infty)^{*}$.  

In more detail, let $\{(G_n,\iota_{nm})\}_{n\in\mathbb{N}}$ be a countable, directed system of complex affine algebraic groups $G_n$  for which the transition maps $\iota_{nm}:\ G_n\hookrightarrow G_m$ are homomorphic embeddings of algebraic groups.  The direct limit group $G:=\dirlim G_{n}$ has the structure of an ind-variety and its Lie algebra $\fg=\dirlim \fg_{n}$ is a direct limit Lie algebra.  The algebraic dual $\fg^{*}=\invlim \fg_{n}^{*}$ is a provariety, an inverse limit in the category of varieties.  We show that $\fg^{*}$ has a natural Poisson structure inherited from the Lie-Poisson structure of each $\fg_{n}^{*}$ and compute its characteristic distribution.   

In infinite dimensions, there is no guarantee that the characteristic distribution of a Poisson manifold is integrable nor that its leaves possess a symplectic structure.  Even in the comparatively well-behaved setting of Banach Lie groups $G$, it is not known whether the coadjoint orbits of $G$ on a predual $\fg_{*}$ of its Lie algebra $\fg$ are weakly symplectic \cite{OR}.  One of the main results of this paper is to show that the coadjoint orbits of a direct limit group 
$G$ on the dual of its Lie algebra $\fg^{*}$ form a symplectic foliation of $\fg^{*}$ which is tangent to the characteristic distribution of $\fg^{*}$.


\noindent
\begin{thm}\label{thm:introthm}
[Proposition \ref{p:symplecticform}, Theorem \ref{thm:symplecticleaves}]\\
{\em Let $G=\dirlim G_{n}$ be a direct limit group, and let $\fg^{*}=\invlim \fg_{n}^{*}$ be the dual space of its Lie algebra.  Let $\lambda\in\fg^{*}$, and let $G\cdot\lambda$ be the coadjoint orbit of $\lambda$.  Then
  \begin{enumerate}
 \item  $G\cdot\lambda$ has the structure of a weak symplectic ind-subvariety of $\fg^{*}$.
  \item $G\cdot \lambda$ is tangent to the characteristic distribution of $\fg^{*}$, and  
  the symplectic structure on $G\cdot\lambda$ is compatible with the Poisson structure on $\fg^{*}$.
  \end{enumerate}
  Thus, the coadjoint orbits of $G$ on $\fg^{*}$ form a weak symplectic foliation of the Poisson provariety $\fg^{*}$.
  }
 \end{thm}
 

  To prove Part (1), we observe that the coadjoint orbit $G\cdot\lambda$ inherits an ind-variety structure from $G$ via:
  $$
  G\cdot\lambda=\dirlim G_{n}\cdot\lambda.
  $$  
  Since $\lambda\in\fg^{*}=\invlim\fg_{n}^{*}$, we can represent $\lambda$ as an infinite sequence, $\lambda=(\lambda_{1},\dots, \lambda_{n},\dots)$ 
  with $\lambda_{n}\in\fg_{n}^{*}$.  For each $n\in\mathbb{N}$, there is a natural projection $G_{n}\cdot\lambda\to G_{n}\cdot \lambda_{n}$, where 
$G_{n}\cdot\lambda_{n}$ is the $G_{n}$-coadjoint orbit of $\lambda_{n}\in\fg_{n}^{*}$.  
Now $G_{n}\cdot\lambda_{n}$ is a symplectic variety with the Kostant-Kirillov symplectic form.  
Pulling back the Kostant-Kirillov form to $G_{n}\cdot\lambda$, we obtain a $2$-form on $G_{n}\cdot\lambda$ for each $n\in\mathbb{N}$.  
We then show that these $2$-forms glue to give a non-degenerate, closed $2$-form on $G\cdot\lambda$ (see Proposition \ref{p:symplecticform}).  To prove Part (2), we use the description of the characteristic distribution of $\fg^{*}$ (see Equations \ref{eq:anchor} and \ref{eq:char}).  

We apply our results to the case where $G$ is the group $GL(\infty):=\dirlim GL(n,\C)$ with the Lie algebra $\fg=\mathfrak{gl}(\infty)=\dirlim \fgl(n,\C)$ of infinite-by-infinite complex matrices with only finitely many nonzero entries.  The dual space $\fg^{*}$ is the Poisson provariety $M(\infty)$ of all infinite-by-infinite complex matrices.
We construct an infinite-dimensional analogue of the Gelfand-Zeitlin integrable system on $M(\infty)$.  
This system generalizes the one constructed by Kostant and Wallach on $\fgl(n,\C)$ in \cite{KW1,KW2}.  

In more detail, we identify $M(\infty)$ with the set of infinite sequences:
$$
M(\infty):=\{X=(X(1), X(2), X(3),\dots ):\, X(n)\in\fgl(n,\C)\mbox { and } X(n+1)_{n}=X(n)\},
$$
where $X(n+1)_{n}$ denotes the $n\times n$ upper left corner of $X(n+1)\in\fgl(n+1,\C)$.  For any $n\in\mathbb{N}$ and $j=1,\dots, n$, let $f_{nj}$ be the function on $M(\infty)$ given by $f_{nj}(X)=tr(X(n)^{j})$, where $tr(\cdot)$ denotes the trace function.  
The algebra generated by the collection of functions $$J_{\infty}:=\{f_{nj}(X) :\, n\in\mathbb{N}, \, j=1,\dots, n\}$$ is then a maximal Poisson-commutative subalgebra of the space of global regular functions on $M(\infty)$.  Moreover, the corresponding Lie algebra of Hamiltonian vector fields $\fa(\infty)$ is infinite dimensional and integrates to an action of a direct limit group $A(\infty)$ which preserves the coadjoint orbits of $GL(\infty)$ on $M(\infty)$.  Following \cite{KW1}, we say that an element $X\in M(\infty)$ is \emph{strongly regular} if the differentials of the functions in $J_{\infty}$ are independent at $X$.  Our main result is the direct limit analogue of Theorem 3.36, \cite{KW1}.
\medskip

\noindent

\begin{thm}\label{thm:sregintro} [Theorem \ref{thm:Lagrangian}]\\
\em{ Let $X\in M(\infty)$ be strongly regular.  Then $A(\infty)\cdot X\subset GL(\infty)\cdot X$ is an irreducible, Lagrangian ind-subvariety 
of $GL(\infty)\cdot X$ with respect to the weak symplectic form on $GL(\infty)\cdot X$ given in Part (1) of Theorem \ref{thm:introthm}.
}

\end{thm}
\medskip

In the finite-dimensional real setting of $\mathfrak{u}(n)$, Guillemin and Sternberg use such a Gelfand-Zeitlin foliation to construct the irreducible representations of the unitary group $U(n)$ via geometric quantization \cite{GS}.  The present paper lays the necessary groundwork and performs the first steps towards an analogous such construction in the direct limit context.  We hope that our work will later lead to a symplectic interpretation of results such as the Bott-Borel-Weil theorem for direct limit groups \cite{DPW}.

The paper is organized as follows.  In Section 2, we study general provarieties $X=\invlim X_{n}$, where $X_{n}$ is a finite-dimensional variety defined over an arbitrary algebraically closed field $F$ of characteristic zero.  We define a structure sheaf $\mathcal{O}_{X}$ which makes the pair $(X,\mathcal{O}_{X})$ into a locally ringed space.    In Section \ref{s:Poisson}, we specialize to the case where each $X_{n}$ is a Poisson variety and show that $\mathcal{O}_{X}$ is a sheaf of Poisson algebras (Definition-Proposition \ref{dp:poisprovariety}).  The provariety structure on $\fg^{*}$ is described in Example \ref{MPoisson}, and its Lie-Poisson structure is obtained in Example \ref{ex:Lie-Poisson}.  In Section 3, we review basic facts about ind-varieties and describe the ind-variety structure of the coadjoint orbits $G\cdot\lambda$ (Proposition \ref{prop:homog} and Corollary \ref{c:coadjointorbit}).  In Section 4, we develop the weak symplectic form on $G\cdot\lambda$ and prove Theorem \ref{thm:introthm}.  In Section 5, we construct the Gelfand-Zeitlin integrable system on $M(\infty)$ and prove Theorem \ref{thm:sregintro}.

\bigskip

\noindent
{\bf Acknowledgements.} We would like to thank Ivan Dimitrov, Sam Evens, Shrawan Kumar, Karl-Hermann Neeb, and Ivan Penkov, for many useful discussions during this project.



\bigskip

\noindent
{\bf Notation.} Throughout this paper, $\mathbb{N}$ and $\C$ will denote the positive integers and complex numbers, respectively.  

\section{Provarieties}
\subsection{The structure sheaf of a provariety}\label{s:provarieties}
Let $\{(X_{n}, p_{nm})\}_{n\in \mathbb{N}}$ be an inverse system of irreducible varieties over an 
algebraically closed field $F$ of characteristic zero with surjective transition morphisms:
$p_{nm}: X_{n}\to X_{m}$ for $n\geq m$.  We call the inverse limit $X=\invlim X_{n}$ of such a system $(X_{n}, p_{nm})$ a {\em provariety}.  
Another introduction to provarieties may be found in \cite{MN}.  They do not assume that their inverse system of varieties is countable.  We will only consider countable inverse systems of varieties, and the exposition here is self-contained.

As a topological space, $X$ has the inverse limit topology.  A basis for this topology is the collection of sets $$\mathcal{B}=\{p_{n}^{-1}(U_{n}): \, U_{n}\subset X_{n} \mbox{ is open}\}.$$  
We construct a structure sheaf $\mathcal{O}_{X}$ on $X$ which makes $(X,\OX)$ into a locally ringed space.  We begin by defining a $\B$-presheaf $\widetilde{\OX}$ of $F$-algebras on $X$, i.e. a presheaf whose sections $\widetilde{\OX}(U)$ are defined only for 
$U\in\B$.  Suppose $U\in\B$ with $U=p_{n}^{-1}(U_{n})$ for some open subset $U_{n}\subseteq X_{n}$.  
 The inverse system $\{(X_{k}, p_{\ell k})\}_{\ell\geq k\geq n}$ gives rise to a directed system of $F$-algebras $\{\mathcal{O}_{X_{k}}(p_{kn}^{-1}(U_{n})), p_{\ell k}^{*}\}_{\ell\geq k\geq n}$.  Since the transition maps $p_{\ell k}$ are surjective for all pairs $\ell\geq k$, it follows that the canonical projections $p_{k}:X\to X_{k}$ are surjective for all $k$. 
Thus, $\mathcal{O}_{X_{k}}(p_{kn}^{-1}(U_{n}))\cong p_{k}^{*}\mathcal{O}_{X_{k}}(p_{kn}^{-1}(U_{n}))$ and we can define:
\begin{equation}\label{eq:Bpresheaf}
\widetilde{\OX}(U):=\dirlim_{k\geq n} p_{k}^{*} \mathcal{O}_{X_{k}}(p_{kn}^{-1}(U_{n})).
\end{equation}
We claim that (\ref{eq:Bpresheaf}) makes $\widetilde{\OX}$ into a $\B$-presheaf.  
Indeed, suppose we have $V\subseteq U$ with $V,\, U\in\B$.  Let $V=p_{\ell}^{-1}(U_{\ell})$ for $U_{\ell}\subseteq X_{\ell}$ open.  
We define the restriction maps $$\rho_{UV}: \widetilde{\OX}(U)\to \widetilde{\OX}(V)$$ as follows.  Suppose $f\in\widetilde{\OX}(U)$.  Then $f=p_{k}^{*} f_{k}$ for some $f_{k}\in \mathcal{O}_{X_{k}}(p_{kn}^{-1}(U_{n}))$ and $k\geq n$.  Let $m\geq \ell, \, k$.  Then $f=p_{m}^{*}p_{mk}^{*} f_{k}$ with $p_{mk}^{*}f_{k}\in\mathcal{O}_{X_{m}}(p_{mn}^{-1}(U_{n}))$. Since $p_{m}$ is surjective, $p_{m\ell}^{-1}(U_{\ell})\subseteq p_{mn}^{-1}(U_{n})$.  
We can therefore define 
$$\rho_{UV}(f):=p_{m}^{*}(p_{mk}^{*} f_{k}|_{p_{m\ell}^{-1}(U_{\ell})})\in\wOX(V),$$ 
where $(p_{mk}^{*} f_{k})|_{p_{m\ell}^{-1}(U_{\ell})}$ denotes the 
restriction of $p_{mk}^{*} f_{k} \in \mathcal{O}_{X_{m}}(p_{mn}^{-1}(U_{n}))$ to $p_{m\ell}^{-1}(U_{\ell})$.  One can verify that $\rho_{UV}$ is well defined and that for 
$W\subseteq V\subseteq U$ with $W\in \B$, we have $\rho_{UW}=\rho_{VW}\circ\rho_{UV}.$
Note also that $\rho_{UU}=Id_{\widetilde{\OX}(U)}$.  Thus, $\widetilde{\OX}$ is a $\B$-presheaf of $F$-algebras.  Since inverse limits exist in the category of $F$-algebras, 
we can form a presheaf on all of $X$ by setting
\begin{equation}\label{eq:Xpresheaf}
\OX(U):=\invlim_{V\subseteq U,\, V\in\B} \widetilde{\OX}(V).
\end{equation}
for each open set $U\subseteq X$.  It follows from the universal property of the inverse limit that $\OX$ is a presheaf on $X$,
and $\OX(U)=\widetilde{\OX}(U)$ for $U\in\B$.  
Moreover,  $\OX$ is in fact a sheaf on $X$.
\begin{prop}\label{p:sheaf}
The presheaf $\OX$ on $X$ is a sheaf of $F$-algebras on $X$.
\end{prop}

\noindent
\begin{proof}
It follows from \cite[Lemma 1.13]{MiDmod} that it suffices to check the sheaf axioms on sections
$\OX(U)$ with $U\in\B$. 
Accordingly, let $U\in\B$ with $U=p_{n}^{-1}(U_{n})$ with $U_{n}\subseteq X_{n}$ open, and let
$\bigcup_{i\in I} p_{i}^{-1}(U_{i})=U$ be an open cover of $U$ by basic open sets of $X$.  
Suppose that for each $i\in I$, we are given $f_{i}\in\OX(p_i^{-1}(U_{i}))$ such that 
\begin{equation}\label{eq:overlap}
f_{i}|_{p_{i}^{-1}(U_{i})\cap p_{j}^{-1}(U_{j})}=f_{j}|_{p_{i}^{-1}(U_{i})\cap p_{j}^{-1}(U_{j})}
\end{equation}
 for every $i, \, j\in I$.  
 Let $F(X_{n})$ be the function field of $X_{n}$.  Consider the field:
\begin{equation}\label{eq:functionfield}
F(X):=\dirlim p_{n}^{*} F(X_{n}).
\end{equation}
 Equation (\ref{eq:overlap}) implies that the functions $f_{i}$ with $i\in I$ define the same element $g\in F(X)$.  
Without loss of generality, we may assume that $g=p_{n}^{*} g_{n}$ for 
$g_{n}\in F(X_{n})$.  We claim that $g_{n}\in\mathcal{O}_{X_{n}}(U_{n})$.  By construction, $g|_{p_{i}^{-1}(U_{i})}=f_{i}$ for all $i$.  
Now let $x\in p_{n}^{-1}(U_{n})$, then $x\in p_{i}^{-1}(U_{i})$ for some $i\in I$.  We have
$$
f_{i}(x)=g(x)=g_{n}(x_{n}), 
$$
where $x_{n}=p_{n}(x)$.  
Since $p_{n}:\invlim X_{k}\to X_{n}$ is surjective,  $g_{n}\in F(X_{n})$ is defined at all points of 
$U_{n}\subset X_{n}$.  Thus, $g_{n}\in \mathcal{O}_{X_{n}}(U_{n})$, so that $g=p_{n}^{*} g_{n}\in \mathcal{O}_{X}(U)$. 

Since the varieties $X_{n}$ are irreducible for all $n$, the restriction maps 
$\rho_{p_{n}^{-1}(U_{n}), p_{i}^{-1}(U_{i})}$ are injective.  Indeed, suppose that $f\in\OX(p_{n}^{-1}(U_{n}))$ 
with $f|_{p_{i}^{-1}(U_{i})}=0$ for some $i\in I$.  Then there exists $k\geq n,\, i$ and a regular function $f_{k}\in \mathcal{O}_{X_{k}}(p_{kn}^{-1}(U_{n}))$
such that $f=p_{k}^{*}f_{k}$ and $f_{k}|_{p_{ki}^{-1}(U_{i})}=0$.  But then since $p_{ki}^{-1}(U_{i})\subseteq p_{kn}^{-1}(U_{n})$ is open and 
$p_{kn}^{-1}(U_{n})$ is irreducible, it follows that $f_{k}=0$ and hence $f=0$.  
\qed

\end{proof}

Proposition \ref{p:sheaf} implies that $(X,\OX)$ is a ringed space.  Since stalks $\stalk$ 
can be computed using basic open sets, Equation \ref{eq:Bpresheaf} implies that
\begin{equation}\label{eq:stalks}
\stalk=\dirlim p_{n}^{*}\mathcal{O}_{X_{n},x_{n}}\cong \dirlim \mathcal{O}_{X_{n},x_{n}}, 
\end{equation}
where $x_{n}=p_{n}(x)$. 
Equation \ref{eq:stalks} implies that $(X,\OX)$ is a locally ringed space.
\begin{prop}\label{p:locrings}
Let $X=\invlim X_{n}$ be a provariety, and let $x\in X$ with $x_{n}=p_{n}(x)$.  Let $\fm_{x_{n}}$ be the unique maximal ideal of the local ring $\mathcal{O}_{X_{n}, x_{n}}$.  Then the stalk $\mathcal{O}_{X,x}$ of the sheaf $\OX$ at $x\in X$ is a local ring with maximal ideal $\mathfrak{m}=\dirlim p_{n}^{*} \mathfrak{m}_{x_{n}}$. 
\end{prop}
The proposition follows immediately from the following general fact. 
\begin{lem}\label{l:localdirect}
Suppose $\{(A_{n},\fm_{n}, \phi_{nm})\}_{n\in\mathbb{N}}$ is a directed system of local rings with $\fm_{n}\subset A_{n}$
the unique maximal ideal and local homomorphisms $\phi_{nm}: A_{n}\to A_{m}$ for $n\leq m$.  Then the direct limit $A=\dirlim A_{n}$
is a local ring with unique maximal ideal $\fm=\dirlim \fm_{n}$.  
\end{lem}

\noindent
\proof Let $\fm=\dirlim\fm_n$ and $a\in A\setminus \fm$.  
Abusing notation, we also denote by $A_{n}$ and $\fm_{n}$, the images of $A_{n}$ 
and $\fm_{n}$ in $\dirlim A_{n}$ respectively.  It follows that $a\in A_{i}\setminus \fm_{i}$ for some $i$, whence $a\in A$ is a unit.  \qed







\subsection{Tangent Spaces to Provarieties}\label{s:tangent}  
Let $(X,\OX)$ be a provariety.  Since $(X,\OX)$ is a locally ringed space, we can define 
the Zariski tangent space $T_{x}(X)$ of a point $x\in X$ as:
\begin{equation}\label{eq:tangentiso}
T_{x}(X):=(\fm_{x}/\fm_{x}^{2})^{*},
\end{equation}
where $(\fm_{x}/\fm_{x}^{2})^{*}$ is the dual of the infinite dimensional $F$-vector space $\fm_{x}/\fm_{x}^{2}$. 
It is easy to see that $T_{x}(X)$ can be identified with the space of all $F$-linear point derivations of the $F$-algebra
$\stalk$ at the point $x\in X$.  The $F$-vector space $T_{x}(X)$ is also an inverse limit.  


\begin{thm}\label{thm:tangentspace}
Let $\{(X_{n},p_{nm})\}$ be an inverse system of varieties, and let $X=\invlim X_{n}$ be the corresponding provariety.  There is a canonical isomorphism 
of $F$-vector spaces: 
\begin{equation}\label{eq:2ndtaniso}
T_{x}(X)\cong \invlim T_{x_{n}}(X_{n}),
\end{equation}
where $x_{n}=p_{n}(x)$ for each $x\in X$.  That is, the following diagram commutes:

\begin{equation}\label{eq:naturality}
\begin{array}{ccc}
\invlim T_{x_{n}}( X_{n}) &\cong & T_{x}(X) \\
\downarrow{\pi_{k}}  &   & \downarrow (dp_{k})_{x} \\
T_{x_{k}}(X_{k}) & = & T_{x_{k}}(X_{k}),
\end{array}
\end{equation}
where $\pi_{k}:\invlim T_{x_{n}}(X_{n})\to T_{x_{k}}(X_{k})$ is the canonical projection.
\end{thm}

\noindent
\proof
Let $x\in X$, and let $x_{n}=p_{n}(x)$ for $n\in\mathbb{N}$.  
The inverse system $\{(X_{n}, p_{nm})\}$ gives rise to an inverse 
system $\{T_{x_{n}}(X_{n}), (dp_{nm})_{x_{n}}\}$.  We can then form the inverse
limit $\invlim T_{x_{n}}(X_{n}).$

By Proposition \ref{p:locrings}, $\fm_{x}=\dirlim p_{n}^{*}\fm_{x_{n}}$, where $\fm_{x_{n}}\subset\mathcal{O}_{X_{n},x_{n}}$ is the unique maximal ideal of $\mathcal{O}_{X_{n},x_n}$.  It follows that $\fm_{x}^{2}=\dirlim p_n^*(\mathfrak{m}_{x_n}^2)$.  Indeed, suppose that $f\in\fm_{x}^{2}$.  Then $f$ is a finite sum $f=\sum_{n,m} (p_{n}^{*} f_{n})(p_{m}^{*} g_{m})$, with $f_{n}\in\fm_{x_{n}}$ and $g_{m}\in\fm_{x_{m}}$.  If we let $\gamma$ be the maximum over all indices $n$ and $m$ appearing in this sum, then $f=\displaystyle\sum_{\mbox{\tiny finite}} p_{\gamma}^{*}( f_{\gamma} g_{\gamma})$, where $f_{\gamma}=p_{\gamma n}^{*} f_{n}$ and $g_{\gamma}=p_{\gamma m}^{*} g_{m}$.  But then $f\in\dirlim p_{n}^{*}(\fm_{x_{n}}^{2})$.  It is easy to see that this argument is independent of the choice of indices used to represent $f$.  Thus, $\mathfrak{m}_{x}^{2}\subseteq \dirlim p_{n}^{*}(\mathfrak{m}_{x_{n}}^{2})$ and the other inclusion is clear.  Therefore,
\begin{equation}\label{eq:tanisos}
\begin{split}
\invlim T_{x_{n}} (X_{n})&=\invlim (\fm_{x_{n}}/\fm_{x_{n}}^{2})^{*}\\
&\cong(\dirlim(\fm_{x_{n}}/\fm_{x_{n}}^{2}))^{*}\\
&\cong (\dirlim\fm_{x_{n}}/ \dirlim\fm_{x_{n}}^{2})^{*}\\
&\cong(\fm_{x}/\fm_{x}^{2})^{*}\\
&=T_{x}(X).
\end{split}
\end{equation}
The commutativity of Diagram (\ref{eq:naturality}) now follows from a simple computation.  \qed

\begin{rem}\label{r:isaprovariety}
If the transition maps $p_{nm}$ are assumed to be surjective submersions for all $n,\, m$, then $T_{x}(X)=\invlim T_{x_{n}}(X_{n})$ has the structure 
of a provariety as in Section \ref{s:provarieties}.  

\end{rem}

\begin{dfn}\label{d:vecfield}
We call a derivation of the sheaf of $F$-algebras $\mathcal{O}_{X}$, $\xi:\OX\to\OX$ a (global) vector field on $X$.  
It follows from definitions that for each $x\in X$, $\xi$ induces a point derivation of the stalk
$\xi_{x}:\stalk\to F$, so that for all $x\in X$, $\xi_{x}\in T_{x}(X)$. 
\end{dfn}

\begin{dfn}\label{dfn:cotangent}
For $x\in X=\invlim X_{n}$, we define the cotangent space at $x$ to be
$$T_{x}^{*}(X):=\dirlim  T_{x_{n}}^{*} (X_{n}),$$ 
where $T_{x_{n}}^{*}(X_{n})$ is the contangent space at $x_{n}=p_{n}(x)$ of $X_{n}$.   
Observe that $(T^{*}_{x}(X))^{*}=T_{x}(X)$ by Theorem \ref{thm:tangentspace}.  
\end{dfn}

\subsection{Morphisms of Provarieties}
In this section, we show that the provariety constructed in 
Section \ref{s:provarieties} is an inverse limit in the category of locally 
ringed spaces.  We first observe that the canonical projection maps: $p_{k}:X=\invlim X_{n}\to X_{k}$ are morphisms of locally ringed spaces with differential $dp_{k}=\pi_{k}: T(X)= \invlim T(X_{n})\to T(X_{k})$.  (See (\ref{eq:naturality})).  This follows immediately from Equation \ref{eq:Bpresheaf}, Proposition \ref{p:locrings}, and Theorem \ref{thm:tangentspace}.  We now examine morphisms in the category of provarieties in more detail.  We begin with the following lemma.
  
\begin{lem}\label{lem:sheaf}
Let $X$ be a topological space, and let $\mathcal{B}$ be a basis for the topology on $X$.  Let $\mathcal{F},\, \mathcal{G}$ be 
sheaves of $F$-algebras on $X$.  Suppose that for any $U\in\mathcal{B}$, we have a homomorphism of $F$-algebras:
$$
\Phi_{U}: \mathcal{F}(U)\to \mathcal{G}(U),
$$
such that if $W\subseteq U$ with $W\in\mathcal{B}$, then the following diagram commutes:
\begin{equation}\label{eq:sheafdiag}
\xymatrixcolsep{75pt}\xymatrix{
\mathcal{F}(U) \ar[r]^{\Phi_{U}}\ar[d]^{\rho_{UW}^{\mathcal{F}}}& \mathcal{G}(U)\ar[d]^{\rho_{UW}^{\mathcal{G}}}\\
\mathcal{F}(W)\ar[r]^{\Phi_{W}}& \mathcal{G}(W).}
\end{equation}
(i.e. $\Phi$ is a morphism of the $\mathcal{B}$-presheaves associated to the sheaves 
$\mathcal{F}$ and $\mathcal{G}$.) Then $\Phi$ lifts to a morphism of sheaves $\tilde{\Phi}:\mathcal{F}\to\mathcal{G}$ 
such that $\widetilde{\Phi}_{U}=\Phi_{U}$ for $U\in\mathcal{B}$.  
\end{lem}

\noindent
\begin{proof}
Let $V\subseteq X$ be open.  Then since $\mathcal{F}$ and $\mathcal{G}$ are sheaves of $F$-algebras,
$$
\mathcal{F}(V)\cong \invlim_{U\subseteq V, U\in\B} \mathcal{F}(U)\mbox{ and }\mathcal{G}(V)\cong \invlim_{U\subseteq V, U\in\B} \mathcal{G}(U).
$$
 Since the diagram in (\ref{eq:sheafdiag}) 
is commutative, the universal property of inverse limits gives a morphism: 
$$
\widetilde{\Phi}_{U}:=\invlim_{U\subseteq V} \Phi_{U}: \mathcal{F}(V)\to\mathcal{G}(V).  
$$
It is easy to see that $\tilde{\Phi}$ is a morphism of sheaves with the desired property. 
\qed
\end{proof}

We now state and prove the main result of this section.  
\begin{prop}\label{prop:provariety}
Let $\{(X_{n},p_{nk})\}$ be an inverse system of varieties with surjective transition morphisms, and let $\mathcal{O}_{X_{n}}$ be the structure sheaf of $X_{n}$.
Let $(X=\invlim X_{n}, \mathcal{O}_{X})$ be the corresponding provariety.
Let $(Y,\mathcal{O}_{Y})$ be a locally ringed space.  Suppose we are given morphisms of locally ringed spaces $\{f_{n}\}_{n\in \mathbb{N}}$ with $f_{n}:(Y,\mathcal{O}_{Y})\to (X_{n},\mathcal{O}_{X_{n}})$ such that for any $m\geq n$ the following diagram 
commutes.
\begin{equation}\label{eq:prodiagram1}
\xymatrix{
(Y,\mathcal{O}_{Y})\ar[d]^{f_{n}}\ar[dr]^{f_m}&\\
(X_{n},\mathcal{O}_{X_{n}})& (X_{m},\mathcal{O}_{X_{m}})\ar[l]^{p_{mn}}.
}
\end{equation}
Then the map $f:=\invlim f_{n}$ is a morphism of locally ringed spaces
$$
f:(Y,\mathcal{O}_{Y})\to (X,\mathcal{O}_{X}).  
$$
Moreover, for any $y\in Y$, the differential $(df)_{y}: T_{y}(Y)\to T_{f(y)}(X)$ is given by:
\begin{equation}\label{eq:prodiff}
(df)_{y}=\invlim (df_{n})_{y}.
\end{equation}
\end{prop}

\noindent
\begin{proof}  Since the diagram in (\ref{eq:prodiagram1}) is commutative, the universal property of inverse limits gives 
us a map of sets $f:=\invlim f_{n}: Y\to \invlim X_{n}=X$.  Since $X$ has the inverse limit topology, 
it follows that $f$ is continuous.  

We claim that $f$ induces a morphism of sheaves of $F$-algebras on $X$, $f^{\sharp}:\mathcal{O}_{X}\to f_{*}\mathcal{O}_{Y}$. 
To show this, we use Lemma \ref{lem:sheaf}.  Let $U\subseteq X$ be a basic open set.  Then $U=p_{n}^{-1}(U_{n})$ for some open set $U_{n}\subseteq X_{n}$.  Since the comorphism $p_{n}^{*}: \mathcal{O}_{X_{n}}\to (p_n)_*\mathcal{O}_{X}$ is injective for all $n$,
the commutativity of Diagram (\ref{eq:prodiagram1}) implies that the following diagram is also commutative:
\begin{equation}\label{eq:fourthprodiag}
\xymatrixcolsep{75pt}\xymatrix{
 &\mathcal{O}_{Y}(f_{n}^{-1}(U_{n}))  \\
p_{n}^{*}\mathcal{O}_{X_{n}}(p_{n}^{-1}(U_{n}))\ar[ur]^{{f_{n}^{\sharp}\circ (p_{n}^{*})^{-1}}}\ar[r]^{p_{m}^{*}\circ p_{mn}^{*}\circ(p_{n}^{*})^{-1}}&  p_{m}^{*}\mathcal{O}_{X_{m}}(p_{mn}^{-1}(U_{n}))\ar[u]_{f_{m}^{\sharp}\circ(p_{m}^{*})^{-1}}.}
\end{equation}
For ease of notation, let $\tilde{f_{m}}:=f_{m}^{\sharp}\circ (p_{m}^{*})^{-1}$ for each $m\in \mathbb{N}$. 
It follows from Diagram (\ref{eq:fourthprodiag}) and the universal property of direct limits that 
\begin{equation}\label{eq:basic}
\dirlim_{m\geq n}\tilde{f_{m}}:\OX(U)=\dirlim_{m\geq n} p_{m}^{*} \mathcal{O}_{X_{m}}(p_{mn}^{-1}(U_{n}))\to \mathcal{O}_{Y}(f_{n}^{-1}(U_{n}))=f_{*}\mathcal{O}_{Y}(p_{n}^{-1}(U_{n}))
\end{equation}
is a homomorphism of $F$-algebras.  It is easy to see that this homomorphism is compatible with restriction maps of the $\B$-presheaf $\wOX$.  Thus, by Lemma \ref{lem:sheaf} we obtain a morphism of sheaves of $F$-algebras:
$$
f^{\sharp}: \mathcal{O}_{X}\to f_{*}\mathcal{O}_{Y}.
$$
Since the maps $f_{n}$ are morphisms of locally ringed spaces, it follows from Proposition \ref{p:locrings} that $(f,f^{\sharp})$ is a morphism of locally ringed spaces.


We now compute the differential of $f$.  The diagram in (\ref{eq:prodiagram1}) gives rise to a commutative diagram
\begin{equation}\label{eq:diffdiag}
\xymatrixcolsep{65pt}\xymatrix{
T_{y}(Y)\ar[d]^{(df_{n})_{y}}  \ar[dr]^{ (df_{m})_{y} }\\
T_{f_{n}(y)}(X_{n})&T_{f_{m}(y)}(X_{m})\ar[l]_{(dp_{mn})_{f_{m}(y)}}.}
\end{equation}
for any $y\in Y$.
It follows from Theorem \ref{thm:tangentspace} that $(df)_{y}=\invlim(df_{n})_{y}: T_{y}(Y)\to \invlim T_{f_{n}(y)}(X_{n})=T_{f(y)}(X)$.
\qed
\end{proof}

\begin{cor}\label{c:morphisms}
Let $\{(X_{n}, p_{nk})_{n\in \mathbb{N}}\}$ and $\{(Y_{n}, q_{nk})\}_{n\in \mathbb{N}}$ 
be inverse systems of varieties with surjective transition morphisms, and let $(X=\invlim X_{n},\OX)$ and 
$(Y=\invlim Y_{n},\mathcal{O}_{Y})$ be the corresponding provarieties.  
Suppose that for each $n\in \mathbb{N}$, we have morphisms 
$f_{n}: X_{n}\to Y_{n}$ such that for any $m\geq n$ the following diagram commutes:
\begin{equation}\label{eq:invdiag}
\xymatrix{
X_{m}\ar[d]_{p_{mn}}\ar[r]^{f_{m}}&  Y_{m}\ar[d]^{ q_{mn} }\\
X_{n}\ar[r]^{f_{n}}&Y_{n}.
}
\end{equation}
Then the map $f=\invlim f_{n}: (X,\OX)\to (Y,\mathcal{O}_{Y})$ is a morphism of locally ringed spaces
with differential: 
$$df=\invlim df_{n}: \invlim T(X_{n})\to \invlim T(Y_{n}).$$
\end{cor}

\noindent
\begin{proof}
The hypotheses of the corollary imply that the maps 
$$\tilde{f_{n}}: \invlim X_{n}\stackrel{p_{n}}{\to} X_{n}\stackrel{f_{n}}{\to} Y_{n} $$
are morphisms of locally ringed spaces satisfying the conditions of Proposition \ref{prop:provariety}.
It then follows that $\invlim\tilde{f_{n}}=\invlim f_{n}$ is a morphism of locally ringed spaces with 
differential $\invlim d\tilde{f_{n}}=\invlim df_{n}$. 
\qed

\end{proof}

\begin{exam}\label{MPoisson}{\em 
For $n\in\mathbb{N}$, let $\fg_{n}$ be a finite dimensional Lie algebra over $\C$.  Suppose we have a chain 
\begin{equation}\label{eq:chain}
\fg_{1}\stackrel{j_{12}}{\to}\fg_{2}\stackrel{j_{23}}{\to}\fg_{3}\to\dots\to \fg_{n}\stackrel{j_{n,n+1}}{\to}\cdots,
\end{equation}
where $j_{n,n+1}:\fg_{n}\to \fg_{n+1}$ is an injective homomorphism of Lie algebras.  The direct limit $\fg:=\dirlim \fg_{n} $ is naturally a Lie 
algebra, and the full vector space dual $\fg^{*}=\invlim \fg_{n}^{*}$ is a provariety.   
For $\lambda\in\fg^{*}$, the tangent space at $\lambda$ is naturally the provariety:
$$
T_{\lambda}(\fg^{*})=\invlim  T_{\lambda_{n}}(\fg_{n}^{*})=\fg^{*} 
$$
by Theorem \ref{thm:tangentspace} and Remark \ref{r:isaprovariety}.  Similarly, we can identify the cotangent space at $\lambda\in\fg^{*}$ with the
Lie algebra $\fg$ as a vector space: 
\begin{equation}\label{eq:cotan}
T_{\lambda}^{*}(\fg^{*})=\dirlim T^{*}_{\lambda_{n}}(\fg_{n}^{*})=\dirlim (\fg_{n}^{*})^{*}\cong \fg.  
\end{equation}

Suppose that for each $n\in \mathbb{N}$, the Lie algebra $\fg_{n}$ is 
reductive with non-degenerate, associative form $\ll\cdot, \cdot\gg $.  
Then we can use the form $\ll \cdot, \cdot\gg $ to identify $\fg_{n}$ 
with $\fg_{n}^{*}$, giving the vector space 
\begin{equation}\label{eq:gtilde}
\tilde{\fg}:=\invlim\fg_{n}
\end{equation}
the structure of a provariety.  By Corollary \ref{c:morphisms}, 
$\tilde{\fg}\cong \fg^{*}$ as provarieties.  

In particular, consider the case where 
$\fg_{n}=\fgl(n,\C)$ is the Lie algebra of $n\times n$ complex matrices.  For $X\in\fg_n$, let $j_{n,n+1}(X)$ be the $(n+1)\times (n+1)$ matrix 
with $(j_{n,n+1}(X))_{kj}=X_{kj}$ for $k, \, j\in\{1,\dots, n\}$ and $(j_{n,n+1}(X))_{kj}=0$ otherwise.  Then $\fg=\fgl(\infty)$ is the Lie algebra
of infinite-by-infinite complex matrices with only finitely many non-zero entries.  Moreover, the Lie algebra $\fg_{n}$ is reductive with 
non-degenerate, associative form $\ll X,Y\gg =tr(XY)$, where $tr(\cdot)$ denotes the trace function.  Using the trace form, 
the map $j_{n,n+1}^{*}:\fg_{n+1}^{*}\to \fg_{n}^{*}$ is identified with the map $p_{n+1, n}:\fg_{n+1}\to \fg_{n}$, where
$p_{n+1,n}(X)=X_{n}$, and $X_{n}$ is the $n\times n$ submatrix in the upper left-hand corner of $X\in\fg_{n+1}$.  
We denote the dual space of $\fg$, $\tilde{\fg}$ defined in Equation (\ref{eq:gtilde}) as $M(\infty)$.  Thus,
\begin{equation}\label{eq:Minfty}
M(\infty):=\{(X(1), X(2), \dots, X(n), X(n+1), \dots ,) :\, X(n)\in\fg_{n}\mbox { and } X(n+1)_{n}=X(n)\}.
\end{equation}
The provariety $M(\infty)$ is naturally isomorphic to the vector space 
of infinite-by-infinite complex matrices with arbitrary entries. 

A similar construction works for any classical direct limit Lie algebra.  For example, if $\fg_{n}=\mathfrak{so}(n,\C)$ is the Lie algebra of 
$n\times n$ complex skew-symmetric matrices, then $\mathfrak{so}(\infty):=\dirlim \fg_{n}$ is the Lie algebra of infinite-by-infinite skew-symmetric matrices with only finitely many nonzero entries.  The dual space $\tilde{\fg}\cong \mathfrak{so}(\infty)^{*}$ is the provariety of infinite-by-infinite complex skew-symmetric matrices.

}
\end{exam}
We will see in the next section that the Lie-Poisson structure of $\fg_{n}^{*}$ has a natural generalization to the 
provariety $\fg^{*}=\invlim \fg_{n}^{*}$.

\subsection{Poisson Provarieties}\label{s:Poisson}

We briefly recall some basic definitions from Poisson geometry.  A variety $X$ is a {\em Poisson variety} if the structure sheaf $\OX$ is a sheaf
of Poisson algebras. That is to say that for each open subset $U\subseteq X$, $\OX(U)$ is a Poisson algebra and the restriction maps
$\rho_{UV}: \OX(U)\to \OX(V)$ are homomorphisms of Poisson algebras.  This is equivalent to specifiyng a regular bivector field $\pi\in \wedge^{2} TX$, whose Schouten-Nijenhuis bracket $[\pi,\pi]=0$.  We have the relation
\begin{equation}\label{eq:bivec}
\{f,g\}(x)= \pi_{x}(df_{x}, dg_{x}), 
\end{equation}
for $x\in X$ and $f,\, g\in\mathcal{O}_{X}(X)$, where $\{\cdot,\cdot\}$ denotes the Poisson bracket on $\mathcal{O}_{X}(X)$. 
 For a regular function $f\in \mathcal{O}_{X}(X)$, we define the Hamiltonian vector field $\xi_{f}$ by 
$$
\xi_{f}(g)=\{f,g\},  
$$
for $g\in\,\mathcal{O}_{X}(X)$.  The Poisson bivector $\pi$ defines a bundle map $\widetilde{\pi}: T^{*} (X)\to T(X)$, given by 
$$
\widetilde{\pi}(\lambda)(\mu)=\pi(\lambda,\mu), 
$$
for $\lambda,\mu\in T^{*}(X)$.  It follows from (\ref{eq:bivec}) that 
$\widetilde{\pi}(df)=\xi_{f}$.  We refer to $\widetilde{\pi}$ as the \emph{anchor map}.  


Given two Poisson varieties $(X_{1},\pi_{1})$, $(X_{2},\pi_{2})$, a morphism $\phi: X_{1}\to X_{2}$ is said to be \emph{Poisson} if the comorphism $\phi^{\sharp}: \mathcal{O}_{X_{2}}\to\phi_{*} \mathcal{O}_{X_{1}}$ is a morphism of sheaves of Poisson algebras.  In particular, we say that $(X_{1}, \pi_{1})\subset (X_{2}, \pi_{2})$ is a Poisson subvariety if the inclusion map $i: X_{1}\to X_{2}$ is Poisson. 

Let $(X_{n},  p_{nm})$ be an inverse system of varieties with surjective transition morphisms.  Suppose that each of the varieties $X_{n}$ is Poisson and the morphisms $p_{nm}: X_{n}\to X_{m}$ are Poisson.  We claim that the structure sheaf $\OX$ constructed in Section \ref{s:provarieties} is a sheaf of Poisson algebras.  
We begin with the following lemma whose proof is elementary.


\begin{lem}\label{l:Poisson}

\begin{enumerate}
\item Let $(A_{n},\phi_{nm})$ be a directed system of Poisson algebras.  That is, $A_{n}$ is a Poisson algebra for each $n$ and $\phi_{nm}: A_{n}\to A_{m}$  is a homomorphism of Poisson algebras for each $n\leq m$.  Then the direct limit $A=\dirlim A_{n}$ has a natural Poisson algebra structure and is a direct limit in the category of Poisson algebras.    
\item Let $(B_{n},\psi_{nm})$ be an inverse system of Poisson algebras.  Then $\invlim B_{n}$ has the structure of a Poisson algebra and is the inverse 
limit in the category of Poisson algebras.  
\end{enumerate}
\end{lem}

\begin{dfn-prop}\label{dp:poisprovariety}  Let $(X_n,p_{nm})$ be an inverse system of Poisson varieties $X_n$, with surjective Poisson morphisms $p_{nm}$, and let $(X=\invlim X_{n}, \OX)$ be the 
corresponding provariety.  Then the structure sheaf $\OX$ constructed in Section \ref{s:provarieties} is a sheaf of Poisson algebras.  We call $X=\invlim X_n$ a {\em Poisson provariety}.
\end{dfn-prop}

\noindent
\begin{proof}
It follows from (\ref{eq:Bpresheaf}) and Part (1) of Lemma \ref{l:Poisson} that the 
$\B$-presheaf $\wOX$ is a $\B$-presheaf of Poisson algebras.  Part (2) of Lemma \ref{l:Poisson} and Equation \ref{eq:Xpresheaf}
then imply that the sheaf $\OX$ is a sheaf of Poisson algebras.  
\qed
\end{proof}

The following lemma will play an important role in the constructions that follow.

\begin{lem}\label{wedgelimlemma} Let $\{V_n,\phi_{nm}\}$ be a directed system of vector spaces.  Then for any $k\in\mathbb{N}$
 $$\left(\bigwedge\!\!{}^k\,\dirlim_{n} V_n\right)^*\cong \invlim_{n}\left[\left(\bigwedge\!\!{}^k\,V_n\right)^*\right].$$
\end{lem}

\noindent
\proof By universal properties of direct limits, there exist $\displaystyle{\phi_\ell:\ V_\ell\rightarrow\dirlim_{n} V_{n}}$ compatible with transition functions $\phi_{\ell m}:\ V_\ell\rightarrow V_m$ for all $\ell\leq m$.  These define maps $\wedge^k\phi_\ell:\ \wedge^k V_\ell\rightarrow \wedge^k\displaystyle\dirlim_{n} V_n$ compatible with the transition maps $\wedge^k\phi_{\ell m}:\ \wedge^k V_\ell\rightarrow \wedge^k V_m$.  This induces a map $\displaystyle{
\dirlim_{n}\wedge^{k}\phi_{n}:\dirlim_{n}\wedge^k V_{n}\rightarrow \wedge^k \dirlim_{n} V_{n}}$.  Dualizing, we obtain the desired map
$$\psi:\ \left(\bigwedge\!\!{}^k\,\dirlim_{n} V_{n}\right)^*\rightarrow\left(\dirlim_{n}\bigwedge\!\!{}^k\,V_n\right)^* =\invlim_{n}\left[\left(\bigwedge\!\!{}^k\,V_n\right)^*\right].$$
It is straightforward to verify that $\psi$ is a vector space isomorphism.  Concretely, $\psi(f)=(f_1,f_2,\ldots)$, where
$$f_n=f \circ\wedge^k\phi_n,$$
for each $n$.\qed

\bigskip

%

Let $X=\invlim X_{n}$ be a Poisson provariety.  As in the finite dimensional case, the Hamiltonian vector field $\xi_{f}$ of $f$ is defined by $\xi_{f}(g)=\{f,g\}$ for any $f,\,g\in\mathcal{O}_{X}(X)$.
The cotangent space $T_{x}^{*}(X)=\dirlim T_{x_{n}}^{*}(X_{n})$ is spanned by 
the differentials $df_{x}$ of global functions $f\in\OX(X)\cong\dirlim\mathcal{O}_{X_{n}}(X_{n})$.  
Thus, for each $x\in X$, the Poisson bracket $\{\cdot,\cdot\}$ defines an element
$\pi_{X,x}\in\left(\wedge^2T_{x}^{*}X\right)^*$ given by 
\begin{equation}\label{eq:bivector}
\pi_{X,x}(df_{x}, dg_{x}):=\{f, g\}(x).
\end{equation}
cf (\ref{eq:bivec}).  By Lemma \ref{wedgelimlemma}, we can view $\pi_{X,x}$ as an element of $\invlim \wedge^2 T_{x_n}X_n$ at each $x\in X$. 
We define the Poisson bivector of $X$, $\pi_{X}$ to be the element of $\invlim \wedge^{2} TX_{n}$ 
whose value at each $x\in X$ is given by (\ref{eq:bivector}).  The bivector $\pi_{X}$ is an inverse limit of the Poisson 
bivector on each $X_{n}$.  

\begin{prop}\label{anchormapprop} Let $X=\invlim X_n$ be a Poisson provariety, and let $\pi_{n}\in \wedge^2 TX_n$ be the bivector fields defining the Poisson structure on $X_n$.  Then $\pi_X=\invlim \pi_{n}\in\invlim\wedge^2 TX_n$.  For each $x\in X$, the anchor map $\widetilde{\pi}_{X,x}:T_{x}^{*}X\rightarrow T_xX$ is 
\begin{equation*}
\widetilde{\pi}_{X,x}(\lambda_{n})=\left(dp_{n1}\wpi_{n,x_{n}}(\lambda_n), dp_{n2}\wpi_{n,x_{n}}(\lambda_n),\ldots,\wpi_{n,x_{n}}(\lambda_n),\wpi_{n+1,x_{n+1}}(dp_{n+1,n}^*\lambda_n),\dots\right),
\end{equation*}
for $\lambda_n\in \dirlim T^{*}_{x_{n}}X_{n}$, a representative of $\lambda_{n}\in T^{*}_{x_{n}}(X_{n})$ in 
the direct limit. 

\end{prop}

\noindent
\begin{proof}
This is an elementary computation using the definition of the Poisson bracket $\{\cdot,\cdot\}$ on $X$. \qed
\end{proof}

\begin{exam}\label{ex:Lie-Poisson}{\em 
For $n\in\mathbb{N}$, let $\fg_{n}$ be a finite dimensional, complex Lie algebra.  Then $\fg_{n}^{*}$ is a Poisson variety with the Lie-Poisson structure.  The Poisson bracket of linear functions $x_{n},\, y_{n}\in\fg_{n}$ is given by their Lie bracket, i.e.
\begin{equation}\label{eq:Lie-Poisson}
\{x_{n}, y_{n}\}(\mu_{n})=\mu_{n}([x_{n},y_{n}]),
\end{equation}
for $\mu_{n}\in\fg_{n}^{*}$ (see for example, Section 1.3, \cite{CG}).  We denote the corresponding bivector by $\pi_{n}\in\wedge^{2}T\fg_{n}^{*}$.  We let $\ad^{*}$ denote the coadjoint action of $\fg_{n}$ on $\fg_{n}^{*}$.  Equation \ref{eq:Lie-Poisson} implies the anchor map $\wpi_{n}$ for the Lie-Poisson structure on $\fg_{n}^{*}$ is given by 
\begin{equation}\label{eq:littleanchor}
\wpi_{n,\mu_{n}}(x_{n})=-\ad^{*}(x_{n})\mu_{n}.
\end{equation}

Now suppose we have a chain of Lie algebras as in Equation \ref{eq:chain} of Example \ref{MPoisson}:
$$
\fg_{1}\stackrel{j_{12}}{\to}\fg_{2}\stackrel{j_{23}}{\to}\fg_{3}\to\dots\to \fg_{n}\stackrel{j_{n,n+1}}{\to}\cdots,
$$
and let $\fg:=\dirlim\fg_{n}$ be the corresponding direct limit Lie algebra.  
Since the homomorphisms $j_{n,n+1}:\fg_{n}\to\fg_{n+1}$ are inclusions, their pullbacks $p_{n+1,n}:\fg_{n+1}^*\to\fg_{n}^*$ are Poisson submersions with respect to the Lie-Poisson structures on $\fg_{n+1}^*$ and $\fg_{n}^*$.  Thus, 
$\fg^{*}=\invlim\fg_{n}^{*}$ is a Poisson provariety with bivector $\pi_{\fg^{*}}=\invlim\pi_{n}$.    

For $\mu\in\fg^{*}$, we identify the cotangent space $T_{\mu}^{*}(\fg^{*})$ with $\fg$ as in (\ref{eq:cotan}). 
Then Proposition \ref{anchormapprop} and Equation \ref{eq:littleanchor} imply that the anchor map is 
\begin{equation}\label{eq:anchor}
\wpi_{\fg^*,\mu}(x_n)=\left(-\ad^*(x_n)\mu_n|_{\fg_{1}},\ldots,-\ad^*(x_n)\mu_n|_{\fg_{n-1}},-\ad^*(x_n)\mu_n,\ldots, -\ad^*(x_n)\mu_k,\ldots\right),
\end{equation}
for $x_{n}\in\fg_{n}\subset \fg$, $\mu\in\fg^*$, and where $-\ad^{*}(x_{n})\mu_{n}|_{\fg_{\ell}}$ denotes the restriction of the linear functional $-\ad^{*}(x_{n})\mu_{n}\in\fg_{n}^{*}$ 
to $\fg_{\ell}$ for $\ell<n$.}
\end{exam}

By Equation \ref{eq:anchor}, the kernel of the anchor map consists precisely of the covectors $x_n\in T_{\mu_{n}}^*\fg_n^*\subseteq T^{*}_{\mu}\fg^*=\fg$ whose coadjoint action $\ad^*(x_n)$ annihilates $\mu_k$ for $k\geq n$.  For $k\geq n$, let $\fg_n^{\mu_k}:=\{x_n\in\fg_n :\ \ad^*(x_n)\mu_k=0\}$ denote the annihilator of $\mu_k$ in $\fg_n$. 
\begin{prop}\label{p:injective}
Let $\mu\in\fg^{*}$ and let $\mbox{Ker } \widetilde{\pi_{\fg^{*}}}_{\mu}$ be the kernel of the anchor map $\widetilde{\pi_{\fg^{*}}}$ at $\mu$.  Then 
\begin{equation}\label{eq:keranchor}
\mbox{Ker }\widetilde{\pi_{\fg^{*}}}_{\mu}=\dirlim_{n}\bigcap_{k\geq n} \fg_{n}^{\mu_{k}}. 
\end{equation}

\end{prop}
In the case where $\fg_{n}$ is reductive with adjoint group $G_{n}$, we can use the non-degenerate $G_{n}$-equivariant form $\ll \cdot,\cdot\gg $ on $\fg_{n}$ to transfer the Lie-Poisson structure of $\fg_{n}^{*}$ to $\fg_{n}$.  The coadjoint action of $G_{n}$ on $\fg_{n}^{*}$ is then identified with 
the adjoint action of $G_{n}$ on $\fg_{n}$.  The induced maps $p_{n+1,n}:\fg_{n+1}\to \fg_{n}$ are Poisson submersions and the provariety $\tilde{\fg}=\invlim\fg_{n}$
defined in Equation \ref{eq:gtilde} is a Poisson provariety.  For example, the provariety $M(\infty)$ defined in Equation \ref{eq:Minfty} is a Poisson provariety. 

Let $X$ be a Poisson provariety with bivector $\pi_{X}\in \invlim\wedge^{2} TX_{n}$.  For each $x\in X$, consider the subspace 
\begin{equation}\label{eq:char}
{\mathfrak{X}}(X)_{x}=\{(\xi_{f})_{x}:\, f\in\mathcal{O}_{X}(X)\}=\mbox{Im}\{\tilde{\pi}_{X,x}(T_{x}^{*}(X))\}\subseteq T_{x}(X).
\end{equation}
As in the finite dimensional case, we refer to the union $\mathfrak{X}(X)=\bigcup_{x\in X} \mathfrak{X}(X)_{x}$ as the \emph{characteristic distribution} of $X$.  

If $(X,\mathcal{O}_{X})$ is a finite dimensional, non-singular, Poisson variety over $\C$ then $\mathfrak{X}(X)$ is an integrable distribution.  Its leaves are immersed Poisson  analytic submanifolds of $(S,\{\cdot, \cdot\}_{S})$ where the Poisson bracket on $S$ is induced by a symplectic form $\omega_{S}$ on $S$ (see Chapter 2, \cite{Va} for example). The Poisson submanifolds $(S,\{\cdot, \cdot\}_{S})$ are referred to as \emph{symplectic leaves} of $X$.  For example, let $\fg_{n}^{*}$ be the dual space of a finite dimensional Lie algebra over $\C$ with 
the Lie-Poisson structure $\pi_{n}$ as in Example \ref{ex:Lie-Poisson}.  Then the symplectic leaves of $(\fg_{n}^{*}, \pi_{n})$
are the coadjoint orbits of $G_{n}$ on $\fg_{n}^{*}$ equipped with Kostant-Kirillov symplectic structure, where $G_{n}$ is any connected Lie group with Lie algebra $\fg_{n}$ (see Proposition 3.1, \cite{Va}). 

In infinite dimensions, it is not known whether the characteristic distribution is integrable even for the case of Banach-Poisson manifolds \cite{OR}.  In Section \ref{s:leaves}, we show that for the dual  $\fg^{*}$ of a direct limit Lie algebra $\fg$, the characteristic distribution is integrable, and that the symplectic foliation of $\fg^{*}$ is given by the coadjoint orbits of an Ind-group $G$ on $\fg^{*}$ with Lie algebra $\fg$.  For this, we need to study ind-varieties and direct limit groups in more detail.


\section{Ind-groups}\label{s:ind-varieties}
\subsection{Basic definitions}

In this section, we recall some basic facts about ind-varieties.  For further reading, see \cite{Ku}.  

For each $n\in\mathbb{N}$, let $X_{n}$ be a finite dimensional variety defined over the field $F$. 
Suppose for any $m\in\mathbb{N}$ with $n\leq m$, we have a locally closed embedding
$i_{nm}: X_{n}\to X_{m}$.  We call the direct limit $X:=\dirlim X_{n}$ of the varieties $\{X_{n}\}_{n\in\mathbb{N}}$, 
an \emph{ind-variety}.\footnote{ The traditional definition of an ind-variety stipulates that the embeddings $i_{nm}: X_{n}\to X_{m}$ are closed (see for example \cite{Ku, DPW}).  We require a slightly more general notion for the objects we consider.} 

  As a topological space $X$ is endowed with the final topology (i.e. the finest topology for which the inclusion maps $\iota_n:\ X_n\hookrightarrow X$ are continuous), so that $U\subset X$ is open if and only if $U\cap X_{n}$ is open for all $n\in\mathbb{N}$.  
It is easy to see that $Z\subset X$ is closed if and only if $Z\cap X_{n}$ is closed for all $n\in\mathbb{N}$.  An ind-variety $X$ is said to be irreducible if its 
underlying topological space is irreducible.  One notes that if 
$X=\dirlim X_{n}$ with $X_{n}$ irreducible for all $n$, then $X$ is irreducible.

For any open set $U\subseteq X$, the structure sheaf is given by $\cO_X(U)=\invlim\cO_{X_n}(U_n)$, where $U_n=U\cap X_n$.  (When there is no ambiguity, we identify $X_n$ with its image $\iota_n(X_n)\subseteq X$.)  A map $f:\ X\rightarrow Y$ is a {\em morphism of ind-varieties} if there is a strictly increasing function $m:\ \mathbb{N}\rightarrow\mathbb{N}$, such that the restriction $f_n$ of $f$ to $X_n\subseteq X$ is a morphism of varieties $f_n:\ X_n\rightarrow Y_{m(n)}$.  The map $f$ induces a morphism of ringed spaces $(X,\cO_X)\rightarrow (Y,\cO_Y)$.  Two ind-variety structures on the same set $X$ are said to be equivalent
if the identity map $I_{X}:X\to X$ is an isomorphism of ind-varieties.  We will
not distinguish between equivalent ind-variety structures. 

The product $X\times Y$ of two ind-varieties $X$ and $Y$ is naturally an ind-variety, by viewing $X\times Y=\dirlim(X\times Y)_n$, where $(X\times Y)_n=X_n\times Y_n$, and the transition maps $\iota_{nm}:\ (X\times Y)_n\rightarrow (X\times Y)_m$ are given by $\iota_{nm}=\iota_{nm}^X\times\iota_{nm}^Y$, where $\iota_{nm}^X:\ X_n\rightarrow X_m$ and $\iota_{nm}^Y:\ Y_n\rightarrow Y_m$ are the corresponding transition maps for $X$ and $Y$.

Given an element $x$ of an ind-variety $X=\dirlim X_n$, there exists $k\in\mathbb{N}$ so that $x\in X_\ell$ for all $\ell\geq k$.  We define the {\em tangent space} $T_x(X)$ to $X$ at $x$ to be $T_x(X)=\displaystyle{\dirlim_{\ell\geq k}T_x(X_\ell)}$.  For a morphism of ind-varieties $f:X\to Y$, the differential $ (df)_{x}$ at $x\in X$ is given by 
$$(df)_x=\dirlim_{\ell\geq k}(df_\ell)_x:\ \dirlim_{\ell\geq k}T_x(X_\ell)\rightarrow\dirlim_{\ell\geq k}T_{f(x)}(Y_{m(\ell)}),$$
where $f_{\ell}: X_{\ell}\to Y_{m(\ell)}$ is the morphism obtained by restricting $f$ to $X_{\ell}$.   

The next proposition asserts that an ind-variety is a direct limit in the category of ringed spaces. 
\begin{prop}\label{prop:indvariety}
Let $\displaystyle{(X=\dirlim X_{n}, \mathcal{O}_{X})}$ be an ind-variety and let $(Y,\mathcal{O}_{Y})$ be a locally ringed space.  For each $i\in\mathbb{N}$, suppose we have morphisms of locally ringed spaces
$$
f_{n}:(X_{n},\mathcal{O}_{X_{n}})\rightarrow (Y,\mathcal{O}_{Y})
$$
such that the following diagram commutes.
\begin{equation}\label{eq:diagram1}
\xymatrixcolsep{75pt}\xymatrix{
(Y,\mathcal{O}_{Y}) &  \\
(X_{n},\mathcal{O}_{X_{n}})\ar[u]_{f_{n}}\ar[r]^{i_{nm}} & (X_{m},\mathcal{O}_{X_{m}})\ar[ul]_{f_{m}}.}
\end{equation}
Then $f:=\displaystyle{\dirlim f_{n}}: (X,\mathcal{O}_{X})\to (Y,\mathcal{O}_{Y})$ is a morphism of ringed spaces with differential:
$df=\displaystyle{\dirlim df_{n}}: T(X)\to T(Y)$.  

\end{prop}
\noindent
\begin{proof}
By the universal property of the direct limit, there is a map of sets $f:=\displaystyle{\dirlim f_{n}: \dirlim X_{n}\to Y}$.  We note that 
$f$ is continuous, since $X=\displaystyle{\dirlim X_{n}=\bigcup_{n\in\mathbb{N}} X_{n}}$ has the final topology and each of the maps 
$f_{n}: X_{n}\to Y$ are continuous.  

We claim that $f$ induces a morphism of sheaves of $F$-algebras on $Y$, $f^{\sharp}:\mathcal{O}_{Y}\to f_{*}\mathcal{O}_{X}$.  The commutative diagram 
in (\ref{eq:diagram1}) gives rises to a commutative diagram of morphisms of sheaves of $F$-algebras on $Y$:
\begin{equation}\label{eq:diagram2}
\xymatrixcolsep{75pt}\xymatrix{
\mathcal{O}_{Y}
\ar[d]_{f_{n}^{\sharp}}   \ar[dr]^{f_{m}^{\sharp}} \\
f_{n,*}\mathcal{O}_{X_{n}} & \ar[l]_{i_{nm}^{\sharp}} f_{m,*}\mathcal{O}_{X_{m}},}
\end{equation}
and $\{f_{m,*}\mathcal{O}_{X_{m}}, i_{nm}^{\sharp}\}$ is an inverse system of sheaves of $F$-algebras on $Y$.  By Exercise II 1.12, \cite{Ha}, $\invlim f_{n,*}\mathcal{O}_{X_{n}}$ is a sheaf of $F$-algebras on $Y$, which satisfies the universal property of inverse limits in the category of sheaves of $F$-algebras on $Y$.  Thus, we get a morphism of sheaves of $F$-algebras on $Y$:
$$
\invlim f_{n}^{\sharp}: \mathcal{O}_{Y}\to \invlim f_{n,*}\mathcal{O}_{X_{n}}.
$$ 
It follows from definitions that $\displaystyle{\invlim f_{n,*} \mathcal{O}_{X_{n}}=f_{*}\mathcal{O}_{X}}$.  If we let $f^{\sharp}:=\displaystyle{\invlim f_{n}^{\sharp}}$, then 
$(f, f^{\sharp}):(X,\mathcal{O}_{X})\to (Y,\mathcal{O}_{Y})$ is a morphism of ringed spaces.

We now compute the differential $df$.  Let $x\in X_{n}\subset X$.  The commutative diagram in (\ref{eq:diagram1}) yields 
a commutative diagram:
\begin{equation}\label{eq:diagram4}
\xymatrixcolsep{75pt}\xymatrix{
T(Y)_{f(x)}\\   
T_{x}(X_{n})\ar[u]_{(df_{n})_{x}} \ar[r]^{(di_{nm})_{x}} & \ar[ul]_{(df_{m})_{x}} T_{x}(X_{m}),}
\end{equation}
By the universal property of direct limits, we obtain a map:
$$
\dirlim_{m\geq n}(df_{m})_{x}: \dirlim_{m\geq n} T_{x}(X_{m})\to T_{f(x)}(Y).
$$
Since $\displaystyle{T_{x}(X)=\dirlim_{m\geq n} T_{x}(X_{m})}$, we have $\displaystyle{(df)_{x}=\dirlim_{m\geq n}(df_{m})_{x}}$.  
\qed

\end{proof}
\subsection{Affine direct limit groups}

Let $\{G_{n}, i_{nm}\}_{m\geq n\in \mathbb{N}}$ be a directed system of affine algebraic groups, and let 
$i_{nm}: G_{n}\to G_{m}$ be a homomorphic embedding of algebraic groups.  Then the image of $G_{n}$ is closed in $G_{m}$ (see for example, 
Section 7.2, \cite{Hum}).  The (affine) direct limit group $G=\dirlim G_{n}$ is then naturally an ind-variety.   

For $G=\dirlim G_{n}$ a direct limit group, we consider the tangent space at the identity, $T_{e}(G)$.  We have $T_{e}(G)=\dirlim T_{e}(G_{n})\cong\dirlim \fg_{n}$, where $\fg_{n}=\hbox{Lie}(G_{n})\cong T_{e} (G_{n})$, and 
we think of $\hbox{Lie}(G_{n})$ as the Lie algebra of right invariant vector fields on $G_{n}$.  The ind-variety $\fg:=\dirlim \fg_{n}=\mbox{Lie}(G)$ 
is a direct limit Lie algebra (see Example \ref{MPoisson}).  
\begin{exam}\label{ex:GLinfty}
\em{For each $n\in\mathbb{N}$, let $G_{n}:=GL(n,\C)$ be 
the group of $n\times n$ invertible matrices over the complex numbers.  
We can embed $G_{n}$ in $G_{n+1}$ via the map
$$
i_{nn+1}: g \hookrightarrow \left[\begin{array}{cc} g & 0\\
0 & 1\end{array}\right]\in G_{n+1}.
$$
This map is clearly a closed embedding, so we can form the direct limit group:
\begin{equation}\label{eq:GLinfty}
GL(\infty):=\dirlim G_{n}.
\end{equation}
Of course, $\mbox{Lie}(GL(\infty))=\fgl(\infty)=\dirlim \fgl(n,\C)$ is the 
direct limit Lie algebra discussed in Example \ref{MPoisson}.}
\end{exam}


An {\em algebraic action} of a direct limit group $G$ on an ind-variety $V$ is a morphism of ind-varieties $f:\ G\times V\rightarrow V$ such that each restriction $f_n=f|_{G_n\times V_n}$ defines an algebraic action of $G_n$ on $V_n$, and the following diagram commutes:
$$
\xymatrix{G_n\times V_n\ar[d]_{\iota_{nm}^G\times\iota_{nm}^V}\ar[rr]^{f_n}&&V_n\ar[d]^{\iota_{nm}^V}\\
G_m\times V_m\ar[rr]_{f_m}&&V_m\;,}
$$  
i.e. $f=\dirlim f_{n}$.  If the algebraic action of $G$ on $V$ is transitive, we say that $V$ is a {\em homogeneous space} for $G$.  If each $V_n$ is a vector space over the base field $F$, then $V$ is an {\em algebraic representation}.  Any algebraic representation
$\rho: G\times V\to V$ induces a representation $d\rho: \fg\times V\to V$ of $\fg$ by differentiation.

\begin{exam}{\em
The adjoint representation $\hbox{Ad}:G\times\fg\rightarrow \fg$ defines an algebraic representation of $G$ on $\fg$, and its differential $\hbox{ad}:\fg\times\fg\rightarrow \fg$ is the adjoint representation of $\fg$:
$$\hbox{ad}(X)(Y)=[X,Y]_k,$$
where $X\in\fg_n$, $Y\in\fg_m$, $k=\hbox{max}\{n,m\}$, and $[X,Y]_k$ 
is the bracket of $X$ and $Y$ thought of as elements of $\fg_{k}$.


The directed system $\iota_{nm}:\fg_n\rightarrow\fg_m$ induces an inverse system $\iota_{nm}^*:\fg_m^*\rightarrow\fg_n^*$ for $n\leq m$.  The transition maps $\iota_{nm}^*:\fg_m^*\rightarrow\fg_n^*$ are $G_n$-equivariant 
with respect to the coadjoint action of $G_{n}\subset G_{m}$ on $\fg_{m}^{*}$ and $\fg_{n}^{*}$.  
Thus, we obtain an action of $G$ on the dual space of its Lie algebra 
$\fg^{*}=\invlim\fg_{n}^{*}$, which we refer to as the coadjoint action of $G$ on $\fg^{*}$.  
 Concretely, let $\lambda=(\lambda_{1},\dots, \lambda_{n},\lambda_{n+1},\dots, )\in\fg^*$.  For $g\in G$, there exists $n>0$ so that $g\in G_{n}$, and then 
\begin{equation}\label{eq:coad}
\Ad^{*}(g)\cdot\lambda=((\Ad^{*}(g)\cdot \lambda_{n})|_{\fg_{1}},\dots, (\Ad^{*}(g)\cdot \lambda_{n})|_{\fg_{n-1}}, \Ad^{*}(g)\cdot \lambda_{n}, \dots, \Ad^{*}(g)\cdot \lambda_{k},\dots),
\end{equation}
where $(\Ad^{*}(g)\cdot \lambda_{n})|_{\fg_{j}}$ denotes the restriction of $\Ad^{*}(g)\cdot \lambda_{n}\in\fg_{n}^{*}$ to $\fg_{j}$ for $j<n$. }
\end{exam}
As has already been discussed in Section \ref{s:Poisson}, $\fg^*$ is a Poisson provariety.  In the next section, we will see that the coadjoint orbits of $G$ on $\fg^{*}$ form a weak symplectic foliation of the Poisson provariety $\fg^{*}$.  To do this, we first need to endow the coadjoint orbits described in (\ref{eq:coad}) with the structure of a $G$-homogeneous ind-variety in a natural way.  The key ingredient is the following proposition. 


\begin{prop}\label{prop:homog}
Let $H$ be a closed subgroup of a direct limit group $G$.  Then $H$ is a direct limit group, and the quotient space $G/H$ is an ind-variety and thus a homogeneous space for $G$.  For any $g\in G$, the tangent space $T_{gH}(G/H)$ can be identified with the ind-variety $\fg/ \Ad(g)\fh$. 


Conversely, if $G$ acts transitively on a nonempty set $X$ and the isotropy group $G^x$ of any $x\in X$ is closed, then $X$ can naturally be given the structure of an ind-variety by identifying $X$ with the $G$-homogeneous ind-variety $G/G^{x}$.  The resulting ind-variety structure on $X$ is independent of the choice 
of point $x\in X$.    

\end{prop}

\noindent
\proof Write $G=\dirlim G_n$ and $H_n=H\cap G_n$ for each $n\in\mathbb{N}$.  Then $H=\dirlim H_n$ is naturally a direct limit subgroup of $G$, and we have the following commutative diagram with exact rows:
$$
\xymatrix{0\ar[r]&H_n\ar[d]_{\iota_{nm}}\ar[r]&G_n\ar[d]_{\iota_{nm}}\ar[r]&G_n/H_n\ar[d]_{\iota_{nm}}\ar[r]&0\\
0\ar[r]&H_m\ar[r]&G_m\ar[r]&G_m/H_m\ar[r]&0}
$$ 
where $\iota_{nm}$ denotes the transition map $G_{n}\to G_{m
}$ as well as its restriction to $H_{n}$ and the induced map on the quotient $i_{nm}: G_n/H_n\to G_{m}/H_{m}$.  
The transition maps $i_{nm}: G_{n}/ H_{n} \to G_{m}/H_{m}$ are locally closed embeddings.  By exactness of the direct limit functor,  the sequence
$$
\xymatrix{0\ar[r]&\dirlim H_n\ar[r]&\dirlim G_n\ar[r]&\dirlim G_n/H_n\ar[r]&0}
$$
is exact, so $\displaystyle{\dirlim G_n/H_n\cong\dirlim G_n/\dirlim H_n=G/H}$
is naturally an ind-variety.  It follows from definitions that the action of $G$ on $G/H$ is algebraic, so that $G/H$ is a $G$-homogeneous space. 

Let $gH\in G/H$ and consider the tangent space $T_{gH}(G/H)$.  By our discussion above, $gH$ can be identified with an unique element 
$g_{n} H_{n}\in \dirlim_{k} G_{k}/ H_{k}.$  It follows that 
\begin{equation*}
T_{gH}(G/H)=\dirlim_{k\geq n} T_{g_{n} H_{k}} (G_{k}/H_{k})=\dirlim_{k\geq n} \fg_{k}/\Ad(g_{n})\fh_{k}\cong \dirlim_{k\geq n}\fg_{k}/ \dirlim_{k\geq n}\Ad(g_{n}) \fh_{k}=\fg/\Ad(g)\fh,
\end{equation*}
where we have used right invariant vector fields to identify the tangent space $T_{xH_{k}}(G_{k}/H_{k})$ with $\fg_{k}/\Ad(x)\fh_{k}$ for any $x\in G_{k}$.

Conversely, suppose that $G$ acts on a nonempty set $X$.  Let $x\in X$.  Then $X=G\cdot x=\bigcup_{n=1}^{\infty} G_{n}\cdot x$.  Since $G^{x}$ is closed, $G_{n}^{x}=G_{n}\cap G^{x}$ is closed for each $n$.  
Thus, $G_{n}\cdot x$ can be given the structure of a variety such that 
$G_{n}\cdot x\cong G_{n}/G_{n}^{x}$ as algebraic varieties. 
Thus, 
$$
X=\dirlim G_{n}\cdot x\cong\dirlim G_{n}/G_{n}^{x}\cong \dirlim G_{n}/ \dirlim  G_{n}^{x}=G/ G^{x}.
$$ 
has the structure of $G$-homogeneous ind-variety.  It is easy to see 
that the choice of any other point $y\in X$ produces an equivalent ind-variety structure on $X$.  

\qed


\bigskip
For a point $\lambda\in\fg^{*}$, we denote its coadjoint orbit by $G\cdot\lambda$.  Using Proposition \ref{prop:homog}, we can endow $G\cdot\lambda$ with the structure of a $G$-homogeneous ind-variety. 
\begin{cor}\label{c:coadjointorbit}
Let $\lambda=(\lambda_{1},\dots, \lambda_{n},\dots, \lambda_{k},\dots )\in\fg^{*}$, with $\lambda_{k}\in\fg_{k}^{*}$, and let $G\cdot\lambda\subset\fg^{*}$ denote the coadjoint orbit through $\lambda$.  
Then the isotropy group of $\lambda$, $G^{\lambda}$ is given by 
\begin{equation}\label{eq:Gniso}
G^{\lambda}=\dirlim G_{n}^{\lambda}\mbox{ where } G_{n}^{\lambda}=G^{\lambda}\cap G_{n}=\bigcap_{k\geq n} G_{n}^{\lambda_{k}},
\end{equation}
where $G_{n}^{\lambda_{k}}$ is the isotropy group 
of $\lambda_{k}\in\fg^{*}_{k}$ under the coadjoint action of 
$G_{n}\subset G_{k}$.

 Thus, $G_{n}^{\lambda}$ is closed, so that 
\begin{equation}\label{eq:coadorbit}
G\cdot\lambda=\dirlim G_{n}\cdot\lambda\cong \dirlim G_{n}/G_{n}^{\lambda}\cong G/G^{\lambda}
\end{equation}
has the structure of a $G$-homogeneous ind-variety.  For any $\mu\in G\cdot\lambda$, we
have
\begin{equation}\label{eq:coadtan}
T_{\mu}(G\cdot\lambda)=\fg/\fg^{\mu}=\dirlim\fg_{n}/\fg_{n}^{\mu}=T_{\mu}(G\cdot\mu),
\end{equation}
where $\fg^{\mu}=Lie(G^{\mu})$ with $G^{\mu}\subset G$ the isotropy group of $\mu$.

\end{cor} 
\begin{proof}
We need only verify that 
\begin{equation}\label{eq:isotropy}
G_{n}^{\lambda}=\bigcap_{k\geq n} G_{n}^{\lambda_{k}}, 
\end{equation}
since the other statements of the corollary then follow immediately from Proposition \ref{prop:homog}.  But (\ref{eq:isotropy}) follows from the definition of the coadjoint action in Equation \ref{eq:coad}.  
\qed  
\end{proof}

\section{Symplectic Foliation of $\fg^{*}$.}\label{s:leaves}

\subsection{Kostant-Kirillov form}  Throughout this section, let $G=\dirlim G_{n}$ be an (affine) direct limit group, with $G_{n}$ a connected, complex, affine algebraic group.  Let $\lambda=(\lambda_{1},\dots, \lambda_{n},\dots,\lambda_{k},\dots )$ be an element of the dual $\fg^{*}=\invlim \fg_{n}^{*}$ of the Lie algebra $\fg=\dirlim\fg_n$ of $G$.  Since $G_{n}$ is connected for each $n$, the coadjoint orbit $G\cdot\lambda=\dirlim G_{n}\cdot\lambda$ is irreducible.  In this section, we develop an analogue of the Kostant-Kirillov form on $G\cdot\lambda$.

\medskip

We now construct a $2$-form on $G\cdot\lambda$.  That is to say, that for each $\mu\in G\cdot\lambda$, we construct an element 
$(\omega_{\infty})_{\mu}\in (\wedge^{2} T_{\mu}(G\cdot\lambda))^{*}$, which 
is closed with respect to a natural exterior derivative on $(\wedge^{2} T(G\cdot\lambda))^{*}$.  
By Equation \ref{eq:coadtan}, it suffices to define $\omega_{\infty}$ at $\mu=\lambda$.  

For each $n\in\mathbb{N}$, we have a natural projection $p_{n}: G_{n}\cdot\lambda\cong G_{n}/G_{n}^{\lambda}\to G_{n}/G_{n}^{\lambda_{n}}\cong G_{n}\cdot\lambda_{n}$, where 
$G_{n}\cdot\lambda_{n}\subset \fg_{n}^{*}$ is the $G_{n}$-coadjoint orbit of $\lambda_{n}\in\fg_{n}^{*}$.  
Consider the diagram 
\begin{equation}\label{eq:orbdiag}
\begin{array}{ccc}
G_{n}\cdot \lambda& \stackrel{\iota_{n,n+1}}{\hookrightarrow} & G_{n+1}\cdot \lambda\\
\downarrow p_{n} &    & \downarrow p_{n+1}\\
G_{n}\cdot \lambda_{n} &  &  G_{n+1}\cdot \lambda_{n+1}.
\end{array}
\end{equation}
The map $p_{n}: G_{n}\cdot \lambda\to G_{n}\cdot \lambda_{n}$ is easily seen to be a surjective submersion with differential at $\lambda\in\fg^{*}$
$$(dp_{n})_{\lambda}: \fg_{n}/\fg_{n}^{\lambda}\to \fg_{n}/\fg_{n}^{\lambda_{n}} \mbox{ given by } (dp_{n})_{\lambda}(X+\fg_{n}^{\lambda})=X+\fg_{n}^{\lambda_{n}}, $$ for $X\in\fg_{n}$.  For $n\in\mathbb{N}$, let $\omega_{n}$ be the Kostant-Kirillov form on the coadjoint orbit $G_{n}\cdot\lambda_{n}$.  We claim that 
\begin{equation}\label{eq:match}
d\iota_{n,n+1}^{*}(dp_{n+1}^{*}\omega_{n+1})_{\lambda}=(dp_{n}^{*}\omega_{n})_{\lambda}.  
\end{equation}
Indeed, let $X+\fg_{n}^{\lambda}, \, Y+\fg_{n}^{\lambda}\in\fg_{n}/\fg_{n}^{\lambda}$.  It is straightforward to verify that
\begin{equation*}
d\iota_{n,n+1}^{*}(dp_{n+1}^{*}\omega_{n+1})_{\lambda}(X+\fg_{n}^{\lambda}, Y+\fg_{n}^{\lambda})=\lambda_{n+1}([X,Y]).
\end{equation*}
Similarly, $dp_n^*\omega_n(X+\fg_{n}^{\lambda}, Y+\fg_{n}^{\lambda})=\lambda_{n}([X,Y])$.  Since $\lambda_{n+1}|_{\fg_{n}}=\lambda_{n}$, these expressions agree.  Thus, by Lemma \ref{wedgelimlemma}, we can define an element of the inverse limit $\invlim \wedge^{2} T^{*}(G_{n}\cdot \lambda)\cong (\wedge^{2} \dirlim T(G_{n}\cdot \lambda))^{*}=(\wedge^{2} T(G\cdot\lambda))^{*}$ by 
\begin{equation}\label{eq:newform}
\omega_{\infty}:=\invlim dp_{n}^{*}\omega_{n}=(dp_1^*\omega_{1},dp_2^*\omega_2,dp_3^*\omega_3,\ldots).  
\end{equation}
By Lemma \ref{wedgelimlemma}, the alternating $k$-forms on $T(G\cdot\lambda)$ can be identified with elements of the space $\displaystyle{\invlim_n\bigwedge\!\!{}^k\,T^*(G_n\cdot\lambda)}$.  We consider the following bicomplex, where $d_{k,n}$ are the exterior derivatives and the $\wedge^k d\iota_{n,n+1}^*$ are obtained from pullbacks of the transition maps $\iota_{n,n+1}:\ G_{n}\rightarrow G_{n+1}$ in the directed system defining $G$:
\begin{equation}\label{eq:bicomplex}
\xymatrix{\vdots\ar[d]_{d_{k-1,1}}&\vdots\ar[d]_{d_{k-1,2}}&\vdots\ar[d]_{d_{k-1,3}}\\
\wedge^kT^*(G_1\cdot\lambda)\ar[d]_{d_{k,1}}&\wedge^kT^*(G_2\cdot\lambda)\ar[l]_{\ \ \wedge^{k} d\iota_{12}^*}\ar[d]_{d_{k,2}}&\wedge^kT^*(G_3\cdot\lambda)\ar[l]_{\ \ \wedge^{k}d\iota_{23}^*}\ar[d]_{d_{k,3}}&\cdots\ar[l]_{\ \ \quad\quad\wedge^{k} d\iota_{34}^*}\\
\wedge^{k+1}T^*(G_1\cdot\lambda)\ar[d]_{d_{k+1,1}}&\wedge^{k+1}T^*(G_2\cdot\lambda)\ar[l]_{\ \ \wedge^{k+1}d\iota_{12}^*}\ar[d]_{d_{k+1,2}}&\wedge^{k+1}T^*(G_3\cdot\lambda)\ar[l]_{\ \ \wedge^{k+1}d\iota_{23}^*}\ar[d]_{d_{k+1,3}}&\cdots\ar[l]_{\ \ \quad\quad\wedge^{k+1}d\iota_{34}^*}\\\wedge^{k+2}T^*(G_1\cdot\lambda)\ar[d]_{d_{k+2,1}}&\wedge^{k+2}T^*(G_2\cdot\lambda)\ar[l]_{\ \ \wedge^{k+2}d\iota_{12}^*}\ar[d]_{d_{k+2,2}}&\wedge^{k+2}T^*(G_3\cdot\lambda)\ar[l]_{\ \ \wedge^{k+2}d\iota_{23}^*}\ar[d]_{d_{k+2,3}}&\cdots\ar[l]_{\ \ \quad\quad\wedge^{k+2}d\iota_{34}^*}\\
\vdots&\vdots&\vdots}
\end{equation}
It is straightforward to verify that all the squares in the bicomplex (\ref{eq:bicomplex}) commute.  Thus, there is a map
$$d_{k,\infty}:\ \invlim_n\bigwedge\!\!{}^k\,T^*(G_n\cdot \lambda)\rightarrow\invlim_n\bigwedge\!\!{}^{k+1}\,T^*(G_n\cdot \lambda),$$
for each $k\geq 0$, given by
\begin{equation}\label{eq:kdiff}
d_{k,\infty}(\alpha_1,\alpha_2,\alpha_3,\ldots)=(d_{k,1}(\alpha_1),d_{k,2}(\alpha_2),d_{k,3}(\alpha_3),\ldots),
\end{equation}
where $(\alpha_1,\alpha_2,\alpha_3,\ldots)\in\displaystyle{\invlim_n\bigwedge^kT^*(G_n\cdot \lambda)}$.  
In particular, $d_{0,\infty}: \mathcal{O}(G\cdot\lambda)\to T^{*}(G\cdot\lambda)$ coincides with the usual notion of the differential for functions
in $\mathcal{O}(G\cdot\lambda)=\invlim \mathcal{O}(G_{n}\cdot\lambda)$.  

The $2$-form $\omega_{\infty}\in (\wedge^{2}T(G\cdot\lambda))^{*}$ induces a map $\widetilde{\omega_{\infty}}: T(G\cdot\lambda)\to T^{*}(G\cdot\lambda)$ 
given by:
\begin{equation}\label{eq:bundlemap}
(\widetilde{\omega_{\infty}})_{\mu}(Y)(Z)=\omega_{\infty,\mu}(Y,Z)\mbox{ for } \mu\in G\cdot\lambda, \, Y,\, Z\in T_{\lambda}(G\cdot\mu)=T_{\mu}(G\cdot\mu).
\end{equation}

Following \cite{OR}, we call $\omega_{\infty}$ a \emph{weak symplectic} form on $G\cdot\lambda$ if the following two conditions are satisfied:

(1) The form $\omega_{\infty}$ is closed with respect to the differential $d_{2,\infty}$ defined in Equation \ref{eq:kdiff}.

(2) For each $\mu\in G\cdot\lambda$, the map $(\widetilde{\omega_{\infty}})_{\mu}$ defined in (\ref{eq:bundlemap}) is an injective, regular linear map 
from the linear ind-variety $\displaystyle{T_{\mu}(G\cdot\mu)=\dirlim T_{\mu}(G_{n}\cdot\mu_{n})}$ to the linear provariety $\displaystyle{T^{*}_{\mu}(G\cdot\mu)=\invlim T^{*}_{\mu}(G_{n}\cdot\mu_{n})}$ (see Remark \ref{r:isaprovariety}).


\begin{prop}\label{p:symplecticform}
For $\lambda\in\fg^{*},$ the coadjoint orbit $(G\cdot\lambda,\omega_{\infty})$ is a weak symplectic ind-variety. If we identify $T_{\lambda}(G\cdot\lambda)\cong \fg/\fg^{\lambda}$, then $\omega_{\infty}$ is given by the formula
\begin{equation}\label{eq:formula}
(\omega_{\infty})_{\lambda}(X+\fg^{\lambda}, Y+\fg^{\lambda})=\lambda([X,Y]),
\end{equation}
for $X,\, Y\in\fg$. 
\end{prop}

\noindent
\begin{proof}
Equation \ref{eq:formula} follows directly from the definition of $\omega_{\infty}$.  We now show that $\omega_{\infty}$ satisfies (2) in the definition of a weak symplectic form.  Without loss of generality, we may assume $\mu=\lambda$.  Consider the map $\omegainftylambdatilde$ defined in (\ref{eq:bundlemap}).  We first show that $\omegainftylambdatilde$ is injective.  Suppose that $X\in\fg_{n}$ is such that $(\omega_{\infty})_{\lambda}(X+\fg_{n}^{\lambda}, Y+\fg_{k}^{\lambda})=0$ for all $Y\in\fg_{k}$ and $k\geq 1$.  For $k\geq n$, Equation \ref{eq:formula} implies that $\lambda_{k}([X,Y])=0$ for all $Y\in\fg_{k}$.  Thus, $X\in\bigcap_{k\geq n} \fg_{n}^{\lambda_{k}}=\fg_{n}^{\lambda}$. 

We now show that $\omegainftylambdatilde$ is a morphism.  By Propositions \ref{prop:provariety} and \ref{prop:indvariety}, $\omegainftylambdatilde$ is a morphism if for every $m,\, k\in\mathbb{N}$, the following composition of maps
$$
\fg_{m}/\fg_{m}^{\lambda}\hookrightarrow \fg/\fg^{\lambda}\stackrel{\omegainftylambdatilde}{\rightarrow}\invlim_{n}T_{\lambda}^{*}(\fg_{n}/\fg_{n}^{\lambda})\stackrel{p_{k}}{\rightarrow} T_{\lambda}^{*}(\fg_{k}/\fg_{k}^{\lambda})
$$
is a morphism of finite dimensional affine varieties.  This is an elementary computation using (\ref{eq:formula}).  

We now show that the 2-form $\omega_{\infty}$ is closed with respect to the differential $d_{2,\infty}$ defined in (\ref{eq:kdiff}).  Since the Kostant-Kirillov form $\omega_n$ on $G_n\cdot\lambda_n$ is closed, we have $d_{n}\omega_n=0$, where $d_n:\ \bigwedge^2T^*(G_n\cdot\lambda_n)\rightarrow\bigwedge^{3}T^*(G_n\cdot\lambda_n)$ is the exterior derivative on $G_n\cdot\lambda_n=G_n/G_n^{\lambda_n}$.  Thus,
\begin{align*}
d_{2,\infty}&=d_{2,\infty}(dp_1^*\omega_1,dp_2^*\omega_2,dp_3^*\omega_3,\ldots)\\
&=(d_{2,1}(dp_1^*\omega_1),\,d_{2,2}(dp_2^*\omega_2),\,d_{2,3}(dp_3^*\omega_3),\,\ldots)\\
&=(dp_1^*(d_1\omega_1),\,dp_2^*(d_2\omega_2),\,dp_3^*(d_3\omega_3),\ldots)\\
&=(dp_1^*(0),dp_2^*(0),dp_3^*(0),\ldots)\\
&=0,
\end{align*}
so the $2$-form $\omega_\infty$ is closed.\qed
\end{proof}


\begin{rem}\label{r:reductive}{\em 
Suppose that $G=\dirlim G_{n}$, where $G_{n}$ is a reductive algebraic group.  Then $\hbox{Lie}(G_{n})=\fg_{n}$ is reductive with a non-degenerate, $\Ad(G_{n})$-invariant form $\ll \cdot,\cdot\gg $, which allows us to identify $\fg_{n}$ with $\fg_{n}^{*}$.  The induced isomorphism $\fg^*\cong\fgtilde=\invlim\fg_n$ is equivariant with respect to the coadjoint action of $G$ on $\fg^*$ and the adjoint action of $G$ on $\fgtilde$:
$$\Ad (g)(x_1,x_2,x_3,\ldots)=\left(\Ad (g)x_n|_{\fg_1},\ldots,\Ad(g)x_n|_{\fg_{n-1}},\Ad (g)x_n,\Ad(g)x_{n+1},\ldots\right),$$
where $g\in G_n$.  In particular, we can transfer the symplectic form on coadjoint orbits in $\fg^*$ to adjoint orbits in $\fgtilde$.
 } 
\end{rem}

In the next theorem, we will consider the inclusion of the coadjoint orbits into the provariety $\fg^*$ and the compatibility of the symplectic structure on coadjoint orbits with the Poisson structure on $\fg^*$.  Consider the natural inclusion
\begin{equation}\label{eq:biginclusion}
i=\dirlim i_n:\ \dirlim G_n\cdot\lambda\hookrightarrow\fg^{*},  
\end{equation}
where $i_{n}:G_{n}/G_{n}^{\lambda}\hookrightarrow \fg^{*}$ is given by 
\begin{equation}\label{eq:n-inclusion}
i_{n}(g_{n}G_{n}^{\lambda})
=(\Ad^{*}(g_{n})\lambda_{n}|_{\fg_{1}},\dots, (\Ad^{*}(g_{n})\lambda_{n})|_{\fg_{n-1}},\Ad^{*}(g_{n})\lambda_{n},\Ad^{*}(g_{n})\lambda_{n+1},\dots).
\end{equation}

Via the map $i$, the coadjoint orbits $G\cdot \lambda=\dirlim G_n\cdot\lambda$ are irreducible, immersed ind-subvarieties that are tangent to the characteristic distribution $\mathfrak{X}(\fg^*)$ defined in (\ref{eq:char}).  More precisely, we have the following theorem. 
\begin{thm}\label{thm:symplecticleaves}
{\bf \rm (1)} \ The natural inclusion $i:\ G\cdot\lambda\hookrightarrow \fg^*$ is an injective immersion of the irreducible ind-variety $G\cdot\lambda$ into the provariety $\fg^*$.

{\bf \rm (2)} \ The coadjoint orbits are tangent to the characteristic distribution of the Poisson provariety $(\fg^*,\pi_{\fg^*})$:
\begin{equation}\label{eq:chardist}
  (di)_{\lambda}(T_{\lambda}(G\cdot\lambda))=\mathfrak{X}(\fg^{*})_{\lambda},
  \end{equation}  

{\bf \rm (3)} \ The symplectic form $\omega_{\infty}$ on $G\cdot\lambda$ is consistent with the Poisson structure of $\fg^{*}$, i.e.
 \begin{equation}\label{eq:consistent} 
 \omega_{\infty, \lambda}(Y,Z)=\bivecfglambda([\anchorfglambda]^{-1}\circ di_{\lambda}(Y), [\anchorfglambda]^{-1}\circ di_{\lambda}(Z))
 \end{equation}
 where $[\anchorfglambda]$ is the bijective morphism $[\anchorfglambda]:\ T_{\lambda}^{*}(\fg^{*})/\mbox{Ker }\anchorfglambda\to \mathfrak{X}(\fg^{*})_{\lambda}$ induced by the anchor map $\anchorfglambda$ (see Equation \ref{eq:anchor}). 
 \end{thm}

\noindent
\begin{proof}
Written explicitly, the inclusion $i_n$ in (\ref{eq:n-inclusion}) is simply the map $\invlim_{j} i_{nj}$, where $i_{nj}:\ G_{n}/G_{n}^{\lambda}\to\fg_{j}^{*}$ is given by
\begin{equation}\label{eq:inj1}
i_{nj}:\ g_{n}G_{n}^{\lambda}\mapsto \Ad^{*}(g_{n})\lambda_{j},
\end{equation}
for all $j\geq n$, and 
\begin{equation}\label{eq:inj2}
i_{nj}:\ g_{n}G_{n}^{\lambda}\mapsto (\Ad^{*}(g_{n})\lambda_{n})|_{\fg_{j}},
\end{equation}
for $j<n$.  By Propositions \ref{prop:indvariety} and \ref{prop:provariety}, the map $i:\ G\cdot\lambda\hookrightarrow\fg^{*}$ is a morphism if 
the maps $i_{jn}$ are morphisms for all $j,n\in\mathbb{N}$.  This follows from the universal property of the geometric quotient $G_{n}/G_{n}^{\lambda}$.   

By Propositions \ref{prop:indvariety} and \ref{prop:provariety}, it follows that the differential 
$$
di:\ T(G\cdot\lambda)=\dirlim_{n} T(G_{n}\cdot\lambda)\to T(\fg^{*})=\invlim_{j} T(\fg_{j}^{*})\mbox{ is precisely } di=\dirlim_{n}\invlim_{j} di_{nj}.
$$
Using Equations \ref{eq:inj1} and \ref{eq:inj2}, we see that $(di_{n})_{\lambda}:\fg_{n}/\fg_{n}^{\lambda}\to T_{\lambda}(\fg^{*})$ is given by 
\begin{equation}\label{eq:diff}
(di_{n})_{\lambda}(X_{n}+\fg_{n}^{\lambda})=(\ad^{*}(X_{n})\lambda_{1}|_{\fg_{1}},\dots, (\ad^{*}(X_{n})\lambda_{n})|_{\fg_{n-1}},\ad^{*}(X_{n})\lambda_{n},\dots, \ad^{*}(X_{n})\lambda_{k},\dots).
\end{equation}
From Equation \ref{eq:diff}, it follows that $(di)_{\lambda}=\dirlim (di_{n})_{\lambda}$ is injective.  Thus, $G\cdot\lambda$ is an immersed ind-subvariety of $\fg^{*}$.


Part (2) follows directly from Equations \ref{eq:anchor} and \ref{eq:diff}.  
Indeed,
\begin{equation*}
\begin{split}
\mathfrak{X}(\fg^{*})_{\lambda}&=\mbox{Im }\wpi_{\fg^{*},\lambda}\\
&=\dirlim_{n}((\ad^{*}(\fg_{n})\lambda_{n}|_{\fg_{1}},\dots, (\ad^{*}(\fg_{n})\lambda_{n})|_{\fg_{n-1}}, \ad^{*}(\fg_{n})\lambda_{n}, \dots, \ad^{*}(\fg_{n})\lambda_{k},\dots)\\
&=di_{\lambda} (T_{\lambda}(G\cdot\lambda)).
\end{split}
\end{equation*}

Finally, we show that (\ref{eq:consistent}) holds.  Without loss of generality, we may assume that
$Y=Y_{m}+\fg_{m}^{\lambda}$ and $Z=Z_{n}+\fg_{n}^{\lambda}$ with $Y_{m}\in\fg_{m}$ , $Z_n\in\fg_n$, and $n\geq m$.  
By (\ref{eq:formula}), the left-hand side of (\ref{eq:consistent}) is 
$\lambda_{n}([Y_{m},Z_{n}]_{n})$, where $[Y_{m}, Z_{n}]_{n}$ denotes the Lie bracket 
of $Y_{m}, Z_{n}$ as elements of $\fg_{n}$.  

To compute the right-hand side of (\ref{eq:consistent}), note that $\mbox{Ker } \anchorfglambda=\fg^{\lambda}$, by (\ref{eq:keranchor}) and  (\ref{eq:Gniso}).  Then (\ref{eq:anchor}) and (\ref{eq:diff}) imply that $[\anchorfglambda]^{-1}\circ di_{\lambda}$ is the identity map on $\fg/\fg^{\lambda}$.
Therefore, 
$$  
\bivecfgmu([\anchorfglambda]^{-1}\circ di_{\lambda}(Y), [\anchorfglambda]^{-1}\circ di_{\lambda}(Z))=\bivecfglambda(Y_{m}+\fg_{m}^{\lambda}, Z_{n}+\fg_{n}^{\lambda}).
$$
Proposition \ref{anchormapprop} implies that $\bivecfglambda=\invlim\pi_{\fg_{n}^{*},\lambda_{n}}$, where
$\pi_{\fg_{n}^{*},\lambda_{n}}$ is the bivector for the Lie-Poisson structure on $\fg_{n}^{*}$ evaluated at $\lambda_{n}$.   
Thus, $\bivecfglambda(Y_{m}+\fg_{m}^{\lambda}, Z_{n}+\fg_{n}^{\lambda})=\lambda_{n}([Y_{m},Z_{n}]_{n})$, and Equation \ref{eq:consistent} holds. 
\qed
\end{proof}

\bigskip

Equation (\ref{eq:consistent}) lets us define Hamiltonian vector fields for functions on 
$G\cdot\lambda$ obtained as pullbacks of functions on $\fg^{*}$, giving a Poisson algebra structure on the set of such functions.
  The following proposition is a restatement of Proposition 7.2 in \cite{OR}.  The proof given there carries 
over to our case. 
\begin{prop}\label{p:newHams}
Let $\lambda\in\fg^{*}$, let $(G\cdot\lambda,\omega_{\infty})$ be the coadjoint orbit 
through $\lambda$, and let $i:(G\cdot\lambda,\omega_{\infty})\hookrightarrow\fg^{*}$ be the inclusion 
morphism given in Equation (\ref{eq:biginclusion}).  
\begin{enumerate}
\item Let $U\subset\fg^{*}$ be open and suppose $\mu\in i^{-1}(U)\subset G\cdot\lambda$.  Let
$f\in \mathcal{O}_{\fg^{*}}(U)$, so that $f\circ i\in \mathcal{O}_{G\cdot\lambda}(i^{-1}(U))$.  The differential
$d(f\circ i)(\mu)\in T^{*}_{\mu}(G\cdot\lambda)$ is given by: 
$$
d(f\circ i)(\mu)=\omega_{\infty,\mu}(di_{\mu}^{-1}(\xi_{f})_{\mu}, \cdot ).
$$
\item  Let $U\subset\fg^{*}$ be open.  Then $i^{*}\mathcal{O}_{\fg^{*}}(U)\subset  \mathcal{O}_{G\cdot\lambda}(i^{-1}(U))$ has 
the structure of a Poisson algebra with Poisson bracket given by:
$$
\{f\circ i, g\circ i\}_{\infty}(\mu):=\omega_{\infty,\mu}((di_{\mu})^{-1}(\xi_{f})_{\mu},(di_{\mu})^{-1}(\xi_{g})_{\mu}).
$$
The pullback $i^{*}:( \mathcal{O}_{\fg^{*}}(U), \{\cdot, \cdot\})\to (i^{*}\mathcal{O}_{\fg^{*}}(U),\{\cdot, \cdot\}_{\infty})$ is a homomorphism of Poisson algebras.  

\end{enumerate}

\end{prop}

We end this section with a discussion of the Lagrangian calculus of a coadjoint orbit $G\cdot \lambda$ that will be useful in Section \ref{ss:Lagrangianfoliation}.

\begin{prop}\label{p:Lag}

Let $\mathcal{L}\subseteq G\cdot\lambda$ be an ind-subvariety, so that $\mathcal{L}=\bigcup_{n=1}^{\infty} \mathcal{L}_n$ with $\mathcal{L}_{n}:= \mathcal{L}\cap (G_{n}\cdot\lambda)$ a locally closed subvariety of $G_{n}\cdot\lambda$.  Let $p_{n}: G_{n}\cdot\lambda\to G_{n}\cdot \lambda_{n}$ be the projection.  
Define $\tilde{\mathcal{L}_{n}}:=p_{n}(\mathcal{L}_{n})$  and suppose that $\tilde{\mathcal{L}_{n}}$ satisfies the following conditions: 
\begin{enumerate}
\item $\tilde{\mathcal{L}_{n}}\subseteq G_{n}\cdot\lambda_{n}$ is a subvariety. 
\item $dp_{n} T(\mathcal{L}_{n})=T(\tilde{\mathcal{L}_{n}})$.
\item$ dp_{n}^{-1}(T(\tilde{\mathcal{L}_{n}}))\subseteq T(\mathcal{L}_{n})$.  
\item $\tilde{\mathcal{L}_{n}}\subseteq (G_{n}\cdot\lambda_{n},\omega_{n})$ is Lagrangian.
\end{enumerate}
Then $\mathcal{L}\subseteq (G\cdot\lambda,\omega_{\infty})$ is a Lagrangian ind-subvariety. 
\end{prop}
\noindent
\begin{proof}  Fix $\mu\in \mathcal{L}$ and let $\ell\geq 1$ be such that $\mu\in G_{\ell}\cdot \lambda$, but $\mu\notin G_{k}\cdot\lambda$ 
for any $k<\ell$.  Note that for any $n\geq \ell$, we have $G_{n}\cdot\lambda=G_{n}\cdot \mu$, so that $T_{\mu}(\mathcal{L})=\dirlim_{n\geq \ell} T_{\mu}(\mathcal{L}_{n})\subset \dirlim_{n\geq \ell} T_{\mu}(G_{n}\cdot\mu)$.  We show that $T_{\mu}(\mathcal{L})=T_{\mu}(\mathcal{L})^{\perp}$, where $T_{\mu}(\mathcal{L})^{\perp}$ denotes the annihilator of $T_{\mu}(\mathcal{L})$ in $T_{\mu} (G\cdot \mu)$ with respect to the weak symplectic form $\omega_{\infty}$.  

We first show that $\mathcal{L}$ is coisotropic, i.e., that $T_{\mu}(\mathcal{L})^{\perp}\subseteq T_{\mu}(\mathcal{L})$.  Let $\xi\in T_{\mu}(\calL)^{\perp}$, with $\xi=\xi_{n}+\fg_{n}^{\mu}$ for some $\xi_{n}\in\fg_{n}$ and $n\geq 1$.  We consider $(\omega_{\infty})_{\mu}(\xi_{n}+\fg_{n}^{\mu}, L_{k}+\fg_{k}^{\mu})$, where $k\geq \ell$.  Suppose $n\leq\ell$.  Then by definition of $\omega_{\infty}$, we have 
$$
(\omega_{\infty})_{\mu}(\xi_{n}+\fg_{n}^{\mu}, T_{\mu}(\mathcal{L}))=(\omega_{\infty})_{\mu}(di_{n\ell}(\xi)+\fg_{\ell}^{\mu},T_{\mu}(\mathcal{L})), 
$$
where $di_{n\ell}$ denotes the differential of the inclusion $i_{n\ell}: G_{n}\cdot\mu\to G_{\ell}\cdot \mu$. We can thus assume, without loss of generality, that $n\geq\ell$.


  Note that $(\omega_\infty)_\mu(\xi_n+\fg_n^\mu,T_\mu(\calL))=0$, so $(\omega_\infty)_\mu(\xi_n+\fg_n^\mu,T_\mu(\calL_k))=0$ for all $k\geq \ell$.  In particular,
\begin{align*}
(\omega_\infty)_\mu(\xi_n+\fg_n^\mu,L_n+\fg_n^\mu)&=(\omega_n)_{\mu_n}(\xi_n+\fg_n^{\mu_n},dp_{n}(L_n+\fg_n^\mu))\\&=0,\end{align*}
for all $L_n+\fg_n^\mu\in T_\mu(\calL_n)$.  By (2), $dp_{n}T_\mu(\calL_n)=T_{\mu_n}(\tilde{\calL}_n)$, so $\xi_n+\fg_n^{\mu_n}\in T_{\mu_n}(\tilde{\calL}_n)^{\perp}$.  By (4), 
$\tilde{\calL}_{n}$ is Lagrangian in $G_{n}\cdot\mu_{n}=G_{n}\cdot\lambda_{n}$, whence $T_{\mu_{n}}(\tilde{\calL_{n}})^{\perp}=T_{\mu_{n}}(\tilde{\calL_{n}}).$ 
 Thus, $$\xi=\xi_n+\fg_n^\mu\in dp_{n}^{-1}(T_{\mu_n}(\tilde{\calL}_n))\subseteq T_\mu(\calL_n)\subseteq T_\mu(\calL)$$ by (3).

We now show that $\calL$ is isotropic.  Suppose that $\xi=\xi_n+\fg_n^\mu\in T_\mu(\calL_n)$ with $n\geq \ell$.  Consider $(\omega_\infty)_\mu(\xi,L_k+\fg_k^\mu)$
for $L_k+\fg_k^\mu\in T_\mu(\calL_k)$, with $k\geq\ell$.  As before, if $k\leq n$, we can identify $L_k+\fg_k^\mu$ with its pushforward in $T_\mu(\calL_n)$.  This lets us reduce to the case where $k\geq n$.  Identifying $\xi$ with its image $d\iota_{nk}(\xi)=\xi_n+\fg_k^\mu\in T_\mu(\calL_k)$, we have 
\begin{align*}
(\omega_\infty)_\mu(\xi,L_k+\fg_k^\mu)&=(\omega_k)_{\mu_k}(dp_{k}(\xi_n+\fg_k^\mu),dp_{k}(L_k+\fg_k^\mu))\\
&=(\omega_k)_{\mu_k}(\xi_n+\fg_k^{\mu_k},L_k+\fg_k^{\mu_k}).
\end{align*}
But $dp_{k}T_\mu(\calL_k)=T_{\mu_k}(\tilde{\calL}_k)$ by (2), and $\tilde{\calL}_k$ is Lagrangian in $G_k\cdot\lambda_k=G_k\cdot\mu_k$.  Thus, $T_{\mu_k}(\tilde{\calL}_k)\subseteq T_{\mu_k}(\tilde{\calL}_k)^\perp$, so $(\omega_\infty)_\mu(\xi,L_k+\fg_k^\mu)=0$.  Since $k\geq\ell$ is arbitrary, we have $T_\mu(\calL)\subseteq T_\mu(\calL)^\perp$, and $\calL$ is Lagrangian.\qed
\end{proof}

\bigskip

\section{Gelfand-Zeitlin Integrable System on $M(\infty)$}
\subsection{The group $A(\infty)$}
In this section, we study the analogue of the Gelfand-Zeitlin\footnote{Alternate spellings of
Zeitlin include Cetlin, Tsetlin, Tzetlin, and Zetlin.  In this paper,
we follow the convention from the earlier work of the first author.} collection 
of functions for the Poisson provariety $M(\infty)$ defined in Example \ref{MPoisson}.  We show that 
the corresponding Lie algebra of Hamiltonian vector fields integrates 
to the action of a direct limit group $A(\infty)$ on $M(\infty)$ whose
generic orbits form Lagrangian ind-subvarieties of the corresponding adjoint orbit.  
We begin by recalling some facts about the Gelfand-Zeitlin integrable system 
on $\fg_{n}=\fgl(n,\C)$ constructed by Kostant and Wallach in \cite{KW1}.  


We denote by $\C[\fg_{n}]$ the polynomial functions on $\fg_{n}$.  For $i=1,\dots, n$ and $j=1,\dots, i$, we let $f_{ij}\in\C[\fg_{n}]$ be the polynomial $f_{ij}(X)=tr(X_{i}^{j})$, where $X_i$ is the $i\times i$ submatrix in the upper left-hand corner of $X\in\fg_{n}$, and $tr(\cdot)$ denotes the trace function on $\fg_{n}$.  If $\C[\fg_{n}]^{G_{n}}$ denotes the $\mbox{Ad}(G_{n})$-invariant polynomials on $\fg_{n}$, then $\C[\fg_{n}]^{G_{n}}$ is the polynomial ring $\C[f_{n1},\dots, f_{nn}]$.  
  Consider the Hamiltonian vector field  $\xi_{\fnij}$ on $\fg_{n}$. 
  For $X\in\fg_{n}$, $(d\fnij)_{X}\in T^{*}_{X}(\fg_{n})=\fg_{n}^{*}$.  We can use the trace form 
  $\ll X,Z\gg =tr(XZ)$ on $\fg_{n}$ to identify the differential 
  $(d\fnij)_{X}$ with an element $\nabla\fnij(X)\in \fg_{n}$.  
  The element $\nabla\fnij(X)$ is determined by its pairing against $Z\in\fg_{n}$ 
  by the formula 
  $$
  \ll \nabla\fnij(X),Z\gg =\frac{d}{dt}|_{t=0} \fnij(X+tZ)=(d\fnij)_{X}(Z).
  $$
  We compute 
  \begin{equation}\label{eq:fdifferential}
 \nabla\fnij(X)=j X_{i}^{j-1}\in\fg_{i}\hookrightarrow\fg_{n},
 \end{equation}
 where $\fg_{i}$ is embedded in the top lefthand corner of $\fg_{n}$ (see Example \ref{MPoisson}).  It follows that 
  \begin{equation}\label{eq:fnijHamiltonian}
  (\xi_{\fnij})_{X}=-[jX_{i}^{j-1}, X]. 
  \end{equation}
  (cf Equation \ref{eq:littleanchor}).  Note that if $i=n$, then $\xi_{f_{nj}}=0$ for all $j=1,\dots, n$, since $f_{nj}\in\C[\fg_{n}]^{G_{n}}$ is a Casimir function for the Lie-Poisson 
  structure on $\fg_{n}$.  
  The Gelfand-Zeitlin collection of functions on $\fg_{n}$ is 
  $J_{GZ}:=\{\fnij:\, 1\leq j\leq i\leq n\}.$  The functions $J_{GZ}$ are Poisson 
  commutative and their restriction to a regular adjoint orbit of $G_{n}$ on $\fg_{n}$
  forms an integrable system \cite{KW1}.

   We let $$\fa(n):=\mbox{span}\{ \xi_{f_{ij}}: 1\leq j\leq i\leq n-1\}$$
be the corresponding Lie algebra of Gelfand-Zeitlin vector fields on $\fg_n$.  Then $\mathfrak{a}(n)$ is an abelian Lie algebra of dimension
$\dn$.  Moreover, the Lie algebra $\fa(n)$ integrates to an analytic action of $A(n):=\C^{\dn}$ on $\fg_{n}$ (see \cite[Section 3]{KW1}). This action can be 
described as follows.  
 We take $\underline{t}=(\underline{t}_{1},\dots, \underline{t}_{n-1})\in\C^{1}\times\dots\times \C^{n-1}=\C^{{n\choose 2}}$ as coordinates on $A(n)$, where $\underline{t}_{i}=(t_{i1},\dots, t_{ii})\in\C^i$ for $1\leq i\leq n-1$.  In these coordinates, the action of $A(n)$ on $\fg_n$ is given by
\begin{equation}\label{eq:action}
a\cdot X=\Ad(\exp(t_{1,1}))\cdot\ldots\cdot\Ad(\exp(jt_{i,j}X_{i}^{j-1}))\cdot\ldots\cdot\Ad(\exp((n-1)t_{n-1,n-1} X_{n-1}^{n-2}))\cdot X,
\end{equation}
 for all $1\leq j\leq i\leq n-1$ and $X\in\fg_n$. 
 Since the Gelfand-Zeitlin functions Poisson commute, $A(n)\cdot X\subset G_{n}\cdot X$ 
 is an isotropic submanifold.  
 For each $X\in\fg_{n}$, it follows from Equation \ref{eq:fnijHamiltonian} that 
\begin{equation}\label{eq:Antangent}
T_{X}(A(n)\cdot X)=\fa(n)_{X}=\mbox{span}\{[X_{i}^{j-1}, X]:\, 1\leq j\leq i\leq n-1\}.  
\end{equation}
 %
 
 We now define an infinite dimensional Gelfand-Zeitlin system $J_\infty$ for the provariety $M(\infty)$ by pulling back the Gelfand-Zeitlin functions $f_{ij}$ to $M(\infty)$.  We recall from Example \ref{MPoisson} that $M(\infty)$ can be identified with the space of sequences:  
 $$
 M(\infty)=\{X=(X(1), X(2), \dots, X(n), X(n+1), \dots ,) :\, X(n)\in\fg_{n}\mbox { and } X(n+1)_{n}=X(n)\}.
 $$  
 We have a natural morphism of locally ringed spaces $p_{n}: M(\infty)\to \fg_{n}$, given by $p_{n}(X)=X(n)$.  Also, the morphism $p_{n}$ is Poisson with respect to the Poisson provariety structure on $M(\infty)$ given in Example \ref{ex:Lie-Poisson} and the Lie-Poisson structure on $\fg_{n}$.  Let  
 $$
 J_{\infty}:=\{p_{n}^{*} f_{nj}:\, n\in\mathbb{N}, \, j=1,\dots, n\}.
 $$
 
 \begin{prop}\label{p:maxcom} 
The set $J_{\infty}$ of Gelfand-Zeitlin functions on $M(\infty)$ is Poisson commutative.  
\end{prop}
\noindent
\proof
For the purposes of this proof, we will denote the Poisson 
bracket on $M(\infty)$ by $\{\cdot,\cdot\}_{\infty}$ and the Poisson bracket on $\fg_{n}$ by $\{\cdot,\cdot\}_{n}$.  Consider $p_{n}^{*}f_{nj}\in J_{\infty}$ for $n\in\mathbb{N}$.  Note that for any $m\leq n$, and any $1\leq k\leq m$, we have $\{p_{n}^{*}f_{nj}, p_{m}^{*} f_{mk}\}_{\infty}=0$.  Indeed, $p_{m}^{*}f_{mk}=p_{n}^{*}p_{nm}^{*} f_{mk}$ so that $\{p_{n}^{*}f_{nj}, p_{m}^{*} f_{mk}\}_{\infty}=\{p_{n}^{*}f_{nj}, p_{n}^{*}p_{nm}^{*}f_{mk}\}_{\infty}=p_{n}^{*}\{f_{nj}, p_{nm}^{*}f_{mk}\}_n$.  But $\{f_{nj}, p_{nm}^{*}f_{mk}\}_n=0$, since elements of $\C[\fg_{n}]^{G_n}$ are Casimir functions for the Lie-Poisson structure on $\fg_{n}$.  A completely analogous argument shows that if $m>n$, $\{p^{*}_{m} f_{mk}, p^{*}_{n} f_{nj}\}_{\infty}=0$. 
\qed

 \begin{rem}\label{r:maxcom}
 In fact, it can be shown that the functions $J_{\infty}$ generate a maximal Poisson commutative subalgebra of $\C[M(\infty)]:=\dirlim \C[\fg_{n}]$.
 \end{rem}
 
We consider the abelian Lie algebra of Hamiltonian vector fields on $M(\infty)$,
\begin{equation}\label{eq:ainf}
\fa(\infty):= \mbox{span}\{\xi_{f} :\, f\in J_\infty\}.
\end{equation}
Let $f=p_{n}^{*}f_{n,j}$, we compute $\xi_{f}$.  
It follows from definitions that $(\xi_{f})_{X}=\widetilde{\pi}_{\infty,X}(dp_{n}^{*}f_{n,j})$, 
where $\widetilde{\pi}_{\infty,X}$ is the anchor map for the Poisson structure $\pi_{\infty}$ 
on $M(\infty)$ evaluated at $X=(X_{1},\dots, X_{n},\dots, X_{k},\dots, )\in M(\infty)$. Equation \ref{eq:anchor} implies that
\begin{equation}\label{eq:bigHamiltonian}
(\xi_{f})_{X}=(0,\dots, 0, \underbrace{-[jX_{n}^{j-1}, X_{n+1}]}_{n+1}, \dots, -[jX_{n}^{j-1}, X_{k}], \dots ).
\end{equation}
We now construct an action of a direct limit group $A(\infty):=\C^{\infty}$ on $M(\infty)$ whose generic orbits on $M(\infty)$ are tangent to the 
Lie algebra of Hamiltonian vector fields $\fa(\infty)$.  

For each $n\in\mathbb{N}$, we have a natural homomorphism
$$
\phi_{n,n+1}: A(n)\hookrightarrow A(n+1)\mbox{ given by } \phi_{n,n+1}(\underline{t}_{1},\dots, \underline{t}_{n-1})=(\underline{t}_{1},\dots, \underline{t}_{n-1}, (0,\dots, 0)). 
$$   
The maps $\phi_{n,n+1}$ are clearly closed embeddings, and thus the direct limit $$A(\infty):=\dirlim A(n)=\bigcup_{n\geq 1} A(n)$$ naturally has the structure 
of a direct limit group.  For each $n\geq 1$,  
it follows from Equation \ref{eq:action} that 
$A(n)$ acts on $M(\infty)$ via: 
\begin{equation}\label{eq:actiononMinfty}
a\cdot X=(X_{1}, t_{11}\cdot X_{2}, \dots, (t_{11}, \dots, t_{n-2,n-2})\cdot X_{n-1},
 (t_{11},\dots, t_{n-1,n-1})\cdot X_{n},\dots, (t_{11},\dots, t_{n-1,n-1})\cdot X_{k},\dots ).
\end{equation}
Observe that the diagram 
\begin{equation*}\label{eq:Andiag}
\xymatrix{
A(n)\times M(\infty)\ar[d]_{\phi_{n,n+1}\times Id}\ar[r]&M(\infty)\ar[d]^{Id}\\
A(n+1)\times M(\infty)  \ar[r]  &M(\infty).
}
\end{equation*}
is commutative, where the horizontal maps are given by (\ref{eq:actiononMinfty}). 
We therefore obtain an action of $A(\infty)$ on $M(\infty)$.  Note that $A(\infty)\cdot X\subseteq i(GL(\infty)\cdot X)\subseteq M(\infty)$.   
However, this is not an algebraic action of $A(\infty)=\C^{\infty}$ on $M(\infty)$.

\subsection{Strongly Regular Orbits}\label{ss:Lagrangianfoliation}

We now show that the generic $A(\infty)$-orbits on $M(\infty)$ form Lagrangian subvarieties of the 
corresponding $GL(\infty)$-adjoint orbit with respect to the symplectic form $\omega_{\infty}$ constructed in Section \ref{s:leaves}.  We first recall
the conditions for an $A(n)$-orbit on $\fgl(n,\C)$ to be generic. 
An element $X\in\fg_{n}$ is said to be $\emph{strongly regular}$ if the differentials 
 $\{(df_{ij})_{X}: 1\leq j\leq i\leq n\}$ are linearly independent (see \cite[Theorem 2.7]{KW1}).  
 We denote the set of strongly regular elements of $\fg_{n}$ by $(\fg_{n})_{sreg}$.  
 
 Strongly regular elements may be characterized as follows:
\begin{prop}\label{prop:sreg}(\cite{KW1}, Proposition 2.7 and Theorem 2.14)
The following statements are equivalent. 
\begin{enumerate}
\item $X\in\fg_{n}$ is strongly regular.

\item The tangent vectors $\{(\xifij)_{X};\, i=1,\dots, n-1,\, j=1,\dots, i\}$ are linearly independent.

\item The elements $X_{i}\in\fg_{i}$ are regular for all $i=1,\dots, n$ and $\fz_{\fg_{i}}(X_{i})\cap\fz_{\fg_{i+1}}(X_{i+1})=0$ for $i=1,\dots, n-1$, where $\fz_{\fg_{i}}(X_{i})$ denotes the centralizer of $X_{i}$ in $\fg_{i}$. 

\item The $A(n)$-orbit of $X$, $A(n)\cdot X$ is a Lagrangian subvariety of the adjoint orbit $G_{n}\cdot X$. In particular, 
$\dim A(n)\cdot X=\dn$. 

\end{enumerate}
\end{prop}

 \begin{rem}\label{r:othersreg}
 For $i=1,\dots, n$, let $Z_{G_{i}}(X_{i})$ denote
 the centralizer in $G_{i}$ of $X_{i}$, so that 
 $Lie(Z_{G_{i}}(X_{i}))=\fz_{\fg_{i}}(X_{i})$. 
 For any $i=1, \dots, n-1$, it is easy to see that 
 $\fz_{\fg_{i}}(X_{i})\cap \fz_{\fg_{i+1}}(X_{i+1})=0$ if and only if
 $Z_{G_{i}}(X_{i})\cap Z_{G_{i+1}}(X_{i+1})=Id_{i+1}$, where $Id_{i+1}$ denotes the $(i+1)\times (i+1)$ identity matrix (see [\cite{CE}, Lemma 5.12]).  
\end{rem}
The notion of strong regularity generalizes naturally to the action of $A(\infty)$ on $M(\infty)$.

\begin{dfn}\label{d:sreg}
{\em We say that $X\in M(\infty)$ is \emph{strongly regular} if the differentials $\{ (df)_{X}: f\in J_{\infty}\}$ are linearly independent at $X$.  We denote the set of strongly regular elements in $M(\infty)$ by $M(\infty)_{sreg}$.}  
\end{dfn}

It is easy to see that $X=(X_{1},\dots, X_{n},\dots)\in M(\infty)_{sreg}$ if and only if $X_{n}\in (\fg_{n})_{sreg}$ for all $n$.  So that we have 
\begin{equation}\label{eq:sregset}
M(\infty)_{sreg}=\invlim (\fg_{n})_{sreg}.
\end{equation}

\begin{rem} 
One can show that $M(\infty)_{sreg}$ is a dense subset of $M(\infty)$ with empty interior. 
\end{rem}

 Using results of the first author, we can easily create examples of strongly regular elements.  
\begin{exam}
\em{
Let $M(\infty)_{\theta}\subseteq M(\infty)$ be the set
\begin{equation}\label{eq:theta}
M(\infty)_{\theta}=\{X\in M(\infty): \sigma(X_{i})\cap \sigma(X_{i+1})=\emptyset \mbox{ for each } i\in\mathbb{N}\},
\end{equation}
where $\sigma(X_{i})$ denotes the spectrum of $X_{i}$.  It follows from \cite[Theorem 5.5]{Col1} that $M(\infty)_{\theta}\subseteq M(\infty)_{sreg}$. } 
\end{exam}

We also have the following characterization of strongly regular elements of $M(\infty)$. 
\begin{prop}\label{p:moresreg}  Let $X=(X_{1},\dots, X_{n},\dots, X_{k}, \dots)\in M(\infty)$.  Then the following conditions are equivalent:
\begin{enumerate}
\item $X$ is strongly regular.
\item For all $i\in\mathbb{N}$, $X_{i}$ is regular and $Z_{G_{i}}(X_{i})\cap Z_{G_{i+1}}(X_{i+1})=Id_{i+1}$.  
\item The tangent vectors $\fa(\infty)_{X}:=\{(\xi_{f})_{X}:f\in J_{\infty}\}$ are linearly independent.
\end{enumerate}
\end{prop}

\noindent
\begin{proof}
The equivalence of statements (1) and (2) follows from Equation \ref{eq:sregset} and part (3) of Proposition \ref{prop:sreg} along with Remark \ref{r:othersreg}.  We now see that (1) is equivalent to (3).
If $X\in M(\infty)_{sreg}$, then for any $n\in \mathbb{N}$, we have $\bigcap_{k\geq n} \fz_{\fg_{k}}(X_{k})=0$ by part (3) of Proposition \ref{prop:sreg}.  Thus, by Proposition \ref{p:injective}, we have that $(\anchori)_{X}$ is injective, which implies (3).   That (3) implies (1) is trivial.    \qed
\end{proof}

\bigskip

Proposition \ref{p:moresreg} and the existence of strongly regular elements immediately imply that the Lie algebra $\fa(\infty)$ is infinite dimensional.  

The main result of this section is the following theorem.
\begin{thm}\label{thm:Lagrangian}
Let $X\in M(\infty)_{sreg}$.  Let $i:GL(\infty)\cdot X\hookrightarrow M(\infty)$ be the inclusion morphism in (\ref{eq:biginclusion}).  Then 
\begin{enumerate}
\item The set $i^{-1}(A(\infty)\cdot X)$ naturally has the structure of an irreducible ind-subvariety of 
$GL(\infty)\cdot X$.  Thus, $A(\infty)\cdot X\subset M(\infty)$ is an immersed irreducible ind-subvariety.
\item The ind-subvariety $i^{-1}(A(\infty)\cdot X)\subseteq GL(\infty)\cdot X$ is Lagrangian with respect to the weak
symplectic form $\omega_{\infty}$ on $GL(\infty)\cdot X$. 
\item For any $Y\in A(\infty)\cdot X$, we have 
\begin{equation}\label{eq:Atangent}
T_Y(A(\infty)\cdot Y)=\fa(\infty)_{Y},
\end{equation}
so that the strongly regular $A(\infty)$-orbits in $M(\infty)$ are tangent to the Lie algebra $\fa(\infty)$ of Hamiltonian vector fields defined 
in Equation (\ref{eq:ainf}).
\end{enumerate}
\end{thm}

\noindent
\begin{proof}
Let $X\in M(\infty)_{sreg}$.  It follows from Part (2) of Proposition \ref{p:moresreg} that for each $n\in\mathbb{N}$, we have 
$\displaystyle{G_n^{X}=\bigcap_{k\geq n} Z_{G_{n}}(X_{k})=Id_{n},}$ where $Id_{n}$ denotes the $n\times n$ identity matrix.  
 Thus, 
$$
GL(\infty)\cdot X=\dirlim G_n/G_n^{X}=\dirlim G_n=GL(\infty).
$$
We claim
\begin{equation}\label{eq:indstructure}
i^{-1}(A(\infty)\cdot X)\cap G_n= i_{n}^{-1}(A(\infty)\cdot X)=Z_{G_{1}}(X_{1})\cdots Z_{G_{n}}(X_{n}),   
\end{equation}
where $i_{n}:G_n\to M(\infty)$ is the morphism in Equation \ref{eq:n-inclusion}, and $$Z_{G_{1}}(X_{1})\cdots Z_{G_{n}}(X_{n})=\{g\in G_n:\; g=z_{1}\cdots z_{n}\mbox{ with } z_{i}\in Z_{G_{i}}(X_{i})\}.$$
For ease of notation, we denote $Z_{G_{1}}(X_{1})\cdots Z_{G_{n}}(X_{n})$ by $\mathcal{Z}_{n}$. 
Indeed, suppose that $g_{n}\in i_{n}^{-1}(A(\infty)\cdot X)$.  Then 
$$
((\Ad(g_{n}) X_{n})_{1},\dots, (\Ad(g_{n}) X_{n})_{n-1}, \Ad(g_{n}) X_{n}, \dots, \Ad(g_{n}) X_{k},\dots, )=a_{m} \cdot X\mbox{ for some } a_{m}\in A(m),
$$
with $m\geq 1$. 
Let $a_{m}=(\underline{t}_{1},\dots, \underline{t}_{m-1})\in \C^{{m\choose 2}}$.  By Equation \ref{eq:action}, $A(m)$ acts on $\fg_{m}$ via 
\begin{equation*}
\begin{split}
&a_{m}\cdot X_{m}=\Ad(h_{m-1}) X_{m},\mbox{ where } h_{m-1}=z_{1}\cdots z_{m-1}\in G_{m-1}, \mbox{ with } \\z_{i}
&=\exp(t_{i1} Id_{i})\cdots \exp(it_{ii} X_{i}^{i-1})\in Z_{G_{i}}(X_{i})\subset G_{i}.
\end{split}
\end{equation*}
First, suppose that $m>n$, and consider $(\Ad(h_{m-1})X_{m})_{n+1}$.  
Since $z_{i}\in Z_{G_{i}}(X_{i})$ for $i=1,\dots, m-1$, it follows that 
$$
(\Ad(h_{m-1})X_{m})_{n+1}=\Ad(z_{1}\dots z_{n}) X_{n+1}=\Ad(g_{n})X_{n+1}. 
$$
Since $X\in M(\infty)_{sreg}$, we have $g_{n}=z_{1}\dots z_{n}$.  The case where $m<n$ is analogous.  
Thus, $i_{n}^{-1}(A(\infty)\cdot X)\subseteq \mathcal{Z}_{n}.$
   
 We now show that $\mathcal{Z}_{n}\subset i_{n}^{-1}(A(\infty)\cdot X)$.   Since $X\in M(\infty)_{sreg}$, $X_{i}$ is regular for all $i$ by Proposition \ref{p:moresreg}.  Whence, $Z_{G_{i}}(X_{i})$ is an abelian, connected algebraic group (Proposition 14, \cite{Ko}).  Therefore, the exponential map, $\exp:\fz_{\fg_{i}}(X_{i})\to Z_{G_{i}}(X_{i})$ 
is a surjective homomorphism of algebraic groups.  It is well-known that $\fz_{\fg_{i}}(X_{i})=\mbox{span} \{ Id_{i},\, \dots, \, X_{i}^{i-1}\}$ for $X_{i}\in\fg_{i}$ regular.  It follows that any $z\in Z_{G_{i}}(X_{i})$ can be written as $z=\exp(t_{i1} Id_{i})\dots \exp( t_{ii} X_{i}^{i-1})$, for some $t_{ij}\in \C$.  The inclusion $\mathcal{Z}_{n}\subseteq i_{n}^{-1}(A(\infty)\cdot X)$ now follows from Equations \ref{eq:action} and \ref{eq:actiononMinfty}.

Now we claim that $\mathcal{Z}_{n}$ is a smooth subvariety of $G_{n}$ of dimension $\dnone$.  It follows from our discussion above that $\Ad(Z_{G_{1}}(X_{1})\cdots Z_{G_{n-1}}(X_{n-1})) X_{n}\subseteq G_{n}\cdot X_{n}$ coincides with the $A(n)$-orbit of $X_{n}$. Moreover, Theorem 3.12, \cite{KW1} implies that $A(n)\cdot X_{n}$ is an irreducible, non-singular variety of dimension $\dn$.  If $p_{n}: G_{n}\to G_{n}\cdot X_{n}$ denotes the orbit map, then Proposition III 10.4, \cite{Ha} implies that $p_{n}$ is a smooth morphism of relative dimension $\dim Z_{G_{n}}(X_{n})=n$.  Since the diagram 
$$
\xymatrix{
\mathcal{Z}_{n}\ar[r]\ar[d]_{p_{n}}  & G_n \ar[d]^{p_{n}}\\
A(n)\cdot X_{n} \ar[r]&  G_n\cdot X_{n}}
$$
is Cartesian, it follows from Proposition III 10.1 (b), \cite{Ha} that $\mathcal{Z}_{n}$ is a non-singular variety of dimension $\dnone$.  Thus, 
$i^{-1}(A(\infty)\cdot X)=\bigcup_{n=1}^{\infty} \mathcal{Z}_{n}$ is an irreducible ind-subvariety of $GL(\infty)\cong GL(\infty)\cdot X$, and 
$A(\infty)\cdot X=i( \bigcup_{n=1}^{\infty} \mathcal{Z}_{n})$ is an irreducible, immersed ind-subvariety of $M(\infty)$.

We now compute the tangent space $T_{z}(\mathcal{Z}_{n})$ for $z\in\mathcal{Z}_{n}$ and show that $i^{-1}(A(\infty)\cdot X)\subset GL(\infty)\cdot X$ is 
Lagrangian.  Represent $z=z_{1}\dots z_{n}$ with $z_{i}\in Z_{G_{i}}(X_{i})$ for $i=1,\dots, n$.  
Let $Y\in M(\infty)$ be given by $Y=\Ad(z) X$, so that $Y_{n}=\Ad(z_{1}\dots z_{n-1})X_{n}$.  
Then $Y\in A(\infty)\cdot X$ and $Y_{n}\in A(n)\cdot X_{n}$.  It follows
from Equation \ref{eq:Antangent} that 
$$
(dp_{n}^{-1})_{z}(T_{Y_{n}}(A(n)\cdot Y_{n}))=\mbox{span}\{ Y_{i}^{j}: 1\leq j\leq i\leq n\}=\mbox{span}\{(d\fnij)_{Y_{n}}: 1\leq j\leq i\leq n\}.
$$
It follows from the definition of strong regularity that $\dim \mbox{span}\{ Y_{i}^{j}: 1\leq j\leq i\leq n\}=\dnone$.  
Since $\dim\mathcal{Z}_{n}=\dnone$ and $\mathcal{Z}_{n}$ is non-singular, we have
\begin{equation}\label{eq:Zntangent}
T_{z}(\mathcal{Z}_{n})=\mbox{span}\{ Y_{i}^{j}: 1\leq j\leq i\leq n\}.
\end{equation}
Part (2) of the theorem now follows immediately from Proposition \ref{p:Lag} and part (4) of Proposition \ref{prop:sreg}.
Part (3) of the theorem follows from Equation \ref{eq:Zntangent} along with Equations \ref{eq:bigHamiltonian} and \ref{eq:diff}.



 \qed

\end{proof}

\bibliographystyle{amsalpha.bst}

\bibliography{bibliography-Lau}

\end{document}